\documentclass[12pt]{article}
\usepackage{latexsym,amsmath,enumerate,amssymb,epsf,graphicx}

\newfont{\bb}{msbm10 at 12pt}
\newfont{\tbb}{msbm10 at 9pt}
\def\l{\hbox{\bb L}}
\def\r{\hbox{\bb R}}
\def\b{\hbox{\bb B}}

\def\h{\hbox{\bb H}}
\def\n{\hbox{\bb N}}

\def\s{\hbox{\bb S}}
\def\ts{\hbox{\tbb S}}
\def\th{\hbox{\tbb H}}

\def\tr{\hbox{\tbb R}}

\def\gH{g_{\mathbb{H}^{n+1}}}

\newcommand{\lsch}{\lambda ({\rm Sch}_g)}

\newcommand{\set}[1]{\left\{#1\right\}}
\newcommand{\meta}[2]{\langle #1,#2 \rangle }

\newcommand{\To}{\longrightarrow }
\newcommand{\Diff}{\text{Con}(\s ^n) }
\newcommand{\IsoH}{\text{Iso}(\h ^{n+1})}

\usepackage[latin1]{inputenc}
\usepackage {amsmath}
\usepackage {amsthm}
\usepackage{times}
\usepackage{amscd}
\usepackage{epsf}

\numberwithin{equation} {section}

\begin{document}

\theoremstyle{plain}\newtheorem{lemma}{Lemma}[section]
\theoremstyle{plain}\newtheorem{proposition}{Proposition}[section]
\theoremstyle{plain}\newtheorem{theorem}{Theorem}[section]
\theoremstyle{plain}\newtheorem{example}{Example}[section]
\theoremstyle{plain}\newtheorem{remark}{Remark}[section]
\theoremstyle{plain}\newtheorem{corollary}{Corollary}[section]
\theoremstyle{plain}\newtheorem{definition}{Definition}[section]
\theoremstyle{plain}\newtheorem{acknowledge}{Acknowledgment}

\begin{center}
\rule{15cm}{1.5pt} \vspace{.6cm}

{\Large \bf Hypersurfaces in Hyperbolic Space with Support Function} 

\vspace{0.9cm}

{\large Vincent Bonini$\, ^\dag$ , Jos\'{e} M. Espinar$\,^\ddag$\footnote{The author is partially supported by Spanish MEC-FEDER Grant MTM2010-19821 and CNPQ-Brazil.} 
and Jie Qing$\,^\star$\footnote{The author is partially supported by NSF DMS-1005295 and CNSF 10728103}}\\

\vspace{0.3cm} \rule{15cm}{1.5pt}
\end{center}

\vspace{.2cm}
\noindent$\mbox{}^\dag$ Department of Mathematics, Cal Poly State University, San Luis Obispo, CA 93407; \\e-mail: vbonini@calpoly.edu \vspace{0.2cm}

\noindent $\mbox{}^\ddag$  Instituto de Matem\'{a}tica Pura e Aplicada, 110 Estrada Dona Castorina, Rio de Janeiro \\ 22460-320, Brazil; 
e-mail: jespinar@impa.br\vspace{0.2cm}

\noindent $\mbox{}^\star$ Department of Mathematics, University of California, Santa Cruz, CA 95064; \\
e-mail: qing@ucsc.edu

\begin{abstract} We develop a global correspondence between immersed horospherically convex hypersurfaces $\phi: {\rm M}^n\to\h^{n+1}$ and complete conformal metrics $e^{2\rho}g_{\ts^n}$ on domains $\Omega$ in the boundary $\s^n$ at infinity of $\h^{n+1}$, where $\rho$ is the horospherical support function,  $\partial_\infty\phi({\rm M}^n) = \partial\Omega$ , and $\Omega$ is the image of the Gauss map $G:{\rm M}^n\to \s^n$. To do so we first establish results on when the Gauss map $G: {\rm M}^n\to \s^n$ is injective. We also discuss when an immersed horospherically convex hypersurface can be unfolded along the normal flow into an embedded one. These results allow us to establish 
general Alexandrov reflection principles for elliptic problems of both immersed hypersurfaces in $\h^{n+1}$ and conformal metrics 
on domains in $\s^n$. Consequently, we are able to obtain, for instance, a strong Bernstein theorem for a complete, immersed, horospherically convex hypersurface in $\h^{n+1}$ of constant mean curvature.
\end{abstract}


\section{Introduction}

In a recent paper \cite{EGM}, the authors observed a very interesting fact that the extrinsic curvature of a horospherically convex hypersurface $\phi: {\rm M}^n\to \h^{n+1}$ ($n\geq 3$)
can be calculated via its horospherical support function $\rho$ as follows:
\begin{equation}\label{curvature-relation}
\lambda_i = \frac 12 - \frac 1{1-\kappa_i}
\end{equation}
where the $\lambda_i$ are the eigenvalues of the Schouten tensor of the horospherical metric $\hat g = e^{2\rho}g_{\ts^n}$ 
and the $\kappa_i$ are the principal curvatures of the hypersurface $\phi$  (see also \cite{Eps2, Eps3}). This observation creates a correspondence that opens a window for more interactions between the study of elliptic problems of Weingarten surfaces in hyperbolic space and the study of elliptic problems of conformal metrics. We will assume throughout the paper that the dimension $n\geq 3$ 
or as stated otherwise. \\

Later it was pointed out in \cite{BEQ} that such correspondence can be seen as the association of a conformal metric at infinity with level surfaces of the geodesic defining functions of the conformal metric. In fact, the level surfaces of the geodesic defining function form the regular part of the normal flow (cf. \cite{Eps3}) of the horospherically convex hypersurfaces both in the hyperbolic metric and the conformally compactified metric. We refer to the part of the normal flow where each leaf is embedded as the regular part. \\

At first the horospherical support function $\tilde\rho$ is defined on the parameter space ${\rm M}^n$ of an immersed horospherically convex hypersurface $\phi: {\rm M}^n\to \h^{n+1}$. 
Hence the so-called horospherical metric  $g_h = e^{2\tilde\rho}G^*g_{\ts^n}$ is originally defined on ${\rm M}^n$ too. It is much more useful if the horospherical support function 
$\tilde\rho$ as well as the horospherical metric $g_h$ can be pushed on a domain in $\s^n$ through the Gauss map $G: {\rm M}^n\to \s^n$.  Indeed, when the Gauss 
map is injective, we may view the hypersurface as a ``graph" of the horospherical support function $\rho = \tilde\rho\cdot G^{-1}$ over the domain $G({\rm M}^n)$ in
$\s^n$.  Though the Gauss map of a compact horospherically convex hypersurface is always injective, the Gauss map of an immersed, 
complete, horospherically convex hypersurface in general may not be injective. \\

We notice that the Gauss map of a horospherically convex hypersurface is naturally a
development map.  Hence, as a consequence of the celebrated injectivity result of Schoen and Yau \cite{SY, SY1}, we obtain the following:

\begin{theorem}\label{quick} Suppose that $\phi: {\rm M}^n\to\h^{n+1}$ is an immersed, complete, horospherically convex hypersurface and suppose that 
\begin{equation}\label{R>=0}
\sum_{i=1}^n\frac 2{1- \kappa_i} \leq n,
\end{equation}
where $\kappa_i$ are principal curvatures of $\phi$. Then its Gauss map is injective.
\end{theorem}

In general, to avoid wild behavior of the end of a horospherically convex hypersurface, we require that  the Gauss map is regular at infinity (cf. Definition \ref{regular-end}).  An immediate consequence of such regularity is the following:

\begin{lemma} Suppose that $\phi: {\rm M}^n\to\h^{n+1}$ is a properly immersed, complete, horospherically convex hypersurface with the Gauss map $G$ regular at infinity. Then 
\begin{equation}
\partial G({\rm M}^n) \subseteq \partial_\infty\phi({\rm M}^n).
\end{equation}
\end{lemma}

Using the uniformly horospherical convexity (cf. Definition \ref{uniform-cv}) to ensure the completeness of the horospherical metric, 
we then establish the following injectivity theorem:

\begin{theorem} Suppose that $\phi: {\rm M}^n\to\h^{n+1}$ is a properly immersed, complete, uniformly horospherically convex hypersurface with the Gauss map regular at
infinity. And suppose that the boundary at infinity $\partial_\infty\phi({\rm M}^n)$ is small in the sense that its Hausdorff dimension is less than $n-2$. Then the 
Gauss map $G: {\rm M}^n\to\s^n$ is injective.
\end{theorem}

One of the most important issues in hypersurface theory is about when an immersed hypersurface is embedded. In contrast to the Hadamard type theorem established in 
\cite{do-warner} (cf. \cite{Had, St}), it is pointed out in \cite{EGM} that even a horospherical ovaloid 
does not have to be embedded. But we observe the following:

\begin{proposition}\label{ovaloid} Suppose that $\phi: {\rm M}^n\to\h^{n+1}, n\geq 2,$ is a compact, immersed, horospherically convex surface. Then $\phi$ can be unfolded into an embedded hypersurface along its normal flow eventually.
\end{proposition} 

Our approach here is to use the connection between normal flows, geodesic defining functions, and conformal metrics at the infinity for the hyperbolic metric $g_{\th^{n+1}}$ observed in \cite{BEQ}. Based on the Hadamard type theorem established in \cite{do-warner} (cf. \cite{Cu, EGR}) we are able to obtain the following:

\begin{theorem}\label{embed} Suppose that $\phi: {\rm M}^n\to\h^{n+1}$ is a properly immersed, complete, uniformly horospherically convex hypersurface with injective Gauss map. In addition, we assume that the boundary at infinity $\partial_\infty\phi({\rm M}^n)$ is a disjoint union of smooth compact submanifolds with no boundary in $\s^n$. Then $\phi$ can be unfolded into an embedded hypersurface along its normal flow eventually.\\

Equivalently, suppose that $e^{2\rho}g_{\ts^n}$ is a complete conformal metric on a domain $\Omega$ in $\s^n$ with bounded curvature. In addition, we assume that the boundary 
$\partial\Omega$ is a disjoint union of smooth compact submanifolds with no boundary in $\s^n$. Then the hypersurface
$$
\phi_t = \frac{e^{\rho+t}}{2}\left( 1+ e^{-2\rho-2t} \left( 1+ |\nabla\rho|^2 \right)\right) (1,x) 
+ e^{-\rho-t} (0, -x +\nabla\rho):\Omega\To\h^{n+1}
$$
is embedded when $t$ is large enough.
\end{theorem}

It is interesting in the surface side to note that one also gets to know the end structure in the proof of the above theorem (cf. Remark \ref{end-structure}). The above embedding theorem is particularly useful when combining with injectivity theorems in this paper and therefore gives us opportunities to apply the Alexandrov refection principle
in dealing with immersed hypersurfaces in hyperbolic space. Based on a slight extension of the Alexandrov-Bernstein theorem in \cite{dCL} we obtain
the following:

\begin{theorem}\label{bernstein} Suppose that $\phi: {\rm M}^n\to \h^{n+1}$ is an immersed, complete, horospherically convex hypersurface with constant mean curvature 
$H = \sum_{i=1}^n\kappa_i$ and
\begin{equation}\label{bernstein-extra}
\sum_{i=1}^n \frac 2{1 - \kappa_i} \leq n.
\end{equation}
Then it is a horosphere if its boundary at infinity is a single point in $\s^n$.
\end{theorem}
 
This is a strong Bernstein theorem for immersed hypersurfaces in hyperbolic space. The condition \eqref{bernstein-extra} is used to apply Theorem \ref{quick} and implies that
$H\geq n$. Similarly, we establish a general Alexandrov reflection principle for immersed, complete, horospherically convex hypersurfaces satisfying general elliptic Weingarten
equations \eqref{equ:weingarten}. \\

Elliptic Weingarten equations (4.8) for hypersurfaces and fully nonlinear elliptic Yamabe type equations (4.4) for conformal metrics have been extensively studied. Both subjects have a long history and both are very important subjects in the fields of differential geometry and partial differential equations. Although they are mostly treated separately, there is a clear indication that these two subjects should be intimately related in terms of the types of problems and the tools that have been used. In this paper we extend the 
correspondence shown in \cite{EGM} and establish the following correspondence between uniformly horospherically convex hypersurfaces and complete conformal metrics with bounded curvature. 

\begin{theorem} Suppose that $\phi: {\rm M}^n\to\h^{n+1}$ is an immersed, complete, uniformly horospherically convex hypersurface with injective Gauss map $G: {\rm M}^n\to\s^n$. Then it induces a complete conformal metric $e^{2\rho}g_{\ts^n}$ on $G({\rm M}^n)\subset\s^n$ with bounded curvature, where $\rho$ is the horospherical support function and
$$
\partial_\infty\phi({\rm M}^n) = \partial G({\rm M}^n).
$$  
On the other hand, suppose that $e^{2\rho}g_{\ts^n}$ is a complete metric on a domain $\Omega$ in $\s^n$ with bounded curvature. Then it induces properly immersed, complete, uniformly horospherically convex hypersurfaces
$$
\phi_t = \frac{e^{\rho+t}}{2}\left( 1+ e^{-2\rho-2t} \left( 1+ |\nabla\rho|^2 \right)\right) (1,x) 
+ e^{-\rho-t} (0, -x +\nabla\rho):\Omega\To\h^{n+1}
$$
and
$$
\partial_\infty\phi_t(\Omega) = \partial\Omega
$$
for $t$ large enough.
\end{theorem}

The correspondence established in the above theorem identifies the problem of finding a properly immersed and complete hypersurface $\phi: {\rm M}^n\to\h^{n+1}$ that satisfies certain geometric equation \eqref{equ:weingarten} with a prescribed boundary at infinity $\partial_\infty\phi({\rm M}^n)$ in $\s^n$ \cite{S, GS} with the problem of finding a complete conformal metric $e^{2\rho}g_{\ts^n}$ that satisfies the corresponding geometric equation \eqref{equ:elliptic} according to \eqref{curvature-relation} on the domain $\Omega\subset\s^n$ whose boundary $\partial\Omega$ is the same as the prescribed boundary at infinity $\partial_\infty\phi({\rm M}^n)$ \cite{CHY, MP}.   For instance, the method of Alexandrov reflection for embedded hypersurfaces in hyperbolic space $\h^{n+1}$ in \cite{A, Ho, dCL} and the method of moving planes (or spheres) in \cite{GNN, CGS} are seen to be the same under the correspondence.  As a consequence of our general Alexandrov reflection principle for horospherically convex hypersurfaces satisfying elliptic Weingarten equations \eqref{equ:weingarten}, we also establish 
a general Alexandrov reflection principle for conformal metrics satisfying fully nonlinear elliptic equations \eqref{equ:elliptic}.  From this general Alexandrov reflection principle, we derive, 
for example,  the following Delaunay type theorem:

\begin{theorem}\label{Liouville-Delaunay}  Suppose that $g$ is a complete conformal metric with bounded Schouten tensor on the domain $\Omega = \s^n\setminus\{p, q\}$. And suppose that $g$ satisfies \eqref{equ:elliptic}. Then $g$ is cylindric.  
\end{theorem}

We would like to remark that this Delaunay type theorem  should be compared with those in \cite{Li, LiLi1, LiLi2}. Their theorems assume the scalar curvature is nonnegative, 
while ours assumes the Schouten tensor is bounded.  \\

To end the introduction we would like to remark that it is not just desirable but imperative for us to consider general fully nonlinear elliptic problems \eqref{equ:elliptic} and \eqref{equ:weingarten} other than, for example,  just the mean curvature equation for hypersurfaces. 
Because, in order to gain the embeddedness and apply the Alexandrov reflection principle, we need to unfold a given hypersurface along the normal flow, in which the 
curvature equation usually does not remain the same. This is seen, for instance, in the proof of Theorem \ref{Th:GeneralizedAlexandrov} in Section 4. \\

The paper is organized as follows: In Section \ref{Sect: local} we recapture the works in \cite{EGM} and \cite{BEQ} and clarify the relation of geodesic defining functions and normal flows. In Section \ref{Sect: global} we develop the global correspondence between admissible hypersurfaces  $\h^{n+1}$ and realizable metrics on domains in $\s^n$. We also prove that an admissible hypersurface can be unfolded into an embedded one along the normal flow when the boundary at infinity is a disjoint union of smooth compact submanifolds with no boundary in $\s^n$. In Section \ref{Sect:Ellipticity} we establish the full correspondence between elliptic problems from the two sides. In particular, we compare Alexandrov theorems with 
Obata theorems, Bernstein theorems with Liouville theorems and even Delaunay type theorems. In fact we extend a general symmetry result in \cite{LR}  for both admissible hypersurfaces and realizable metrics based on our embedding theorem.

\vskip 0.2in\noindent{\bf Acknowledgment} \quad The authors would like to express their gratitude to the Beijing International Center for Mathematical Research. Part of the research of this paper was carried out during the time when the authors were visiting the center. The authors are also very appreciative for many enlightening discussions with Professors 
Sun-Yung Alice Chang, Jos\'{e} A. G\'{a}lvez, Bo Guan, Yanyan Li, and Paul Yang.  


\section{Local Thoery}\label{Sect: local}

In this section we will recapture the works in \cite{EGM, BEQ} and set the stage to develop a global theory of the correspondence between hypersurfaces in hyperbolic space 
$\h^{n+1}$ and conformal metrics on domains of the conformal infinity $\s^n$ of hyperbolic space $\h^{n+1}$. In \cite{EGM}, Espinar, G\'{a}lvez and Mira discovered that (see also \cite{Eps2, Eps3}), given a piece of horospherically convex hypersurface $\phi: \rm{M}^n\to\h^{n+1}$, there is a locally conformally flat metric $g_h$ on $\rm{M}^n$, 
whose curvature is explicitly related to the extrinsic curvature of the hypersurface in $\h^{n+1}$. Conversely, one may construct an immersed, horospherically convex hypersurface in hyperbolic space 
$\h^{n+1}$ from a conformal metric on a domain in the infinity $\s^n$.  It was later observed in \cite{BEQ} that such correspondence can be seen as the association of conformal metrics 
on domains of $\s^n$,  geodesic defining functions, and level surfaces of geodesic defining functions (see also \cite{MP2}).


\subsection{Horospherical Convexity and Horospherical Metrics}\label{Sect:horom}

We will briefly introduce the construction developed in \cite{EGM}.
Let us denote by $\r^{1,n+1}$ the Minkowski spacetime, that is,
the vector space $\r ^{n+2}$ endowed with the Minkowski spacetime metric 
$\meta{}{}$ given by
$$ 
\meta{\bar{x}}{\bar{x}} = - x_0 ^2 + \sum _{i=1}^{n+1} x_i ^2 ,
$$
where $\bar{x} \equiv (x_0 , x_1 , \ldots , x_{n+1})\in \r^{n+2}$. Then hyperbolic space, the de Sitter spacetime, and the positive null cone are given, respectively, by the hyperquadrics
\begin{equation*}
\begin{split}
\h ^{n+1} &= \set{ \bar{x} \in \l ^{n+2} : \, \meta{\bar{x}}{\bar{x}} = -1, \, x_0 >0}\\
\s^{n+1}_1 &= \set{ \bar{x} \in \l ^{n+2} : \, \meta{\bar{x}}{\bar{x}} = 1}\\
\n^{n+1}_+ &= \set{ \bar{x} \in \l ^{n+2} : \, \meta{\bar{x}}{\bar{x}} = 0, \, x_0
>0}.
\end{split}
\end{equation*}
The ideal boundary at infinity of hyperbolic space $\h^{n+1}$ will be denoted by $\s ^n$.

An immersed hypersurface in hyperbolic space $\h^{n+1}$ is given by a parametrization
$$
\phi: \rm{M}^n\To\h^{n+1}.
$$
On the hypersurface $\phi$, an orientation assigns a unit normal vector field
$$
\eta: \rm{M}^n\To \s^{n+1}_1.
$$
Hence, associated to $\phi$, one may consider the map
$$
\psi=\phi+\eta: \rm{M}^n\To\n^{n+1}_+,
$$
which is called the associated light cone map of $\phi$. We will use horospheres to define the Gauss map of an oriented, immersed hypersurface in hyperbolic space $\h^{n+1}$. In the above hyperboloid model, horospheres in $\h^{n+1}$ are the intersections of affine null hyperplanes of $\r^{1, n+1}$ with $\h^{n+1}$. 

\begin{definition}[\cite{Eps1,Eps2,Br}]\label{h-gauss}
Let $\phi:\rm{M}^n\To \h ^{n+1}$ be an immersed, oriented hypersurface in $\h^{n+1}$
with the orientation $\eta: \rm{M}^n\to \s^{n+1}_1$. The \emph{Gauss map} 
$$
G: \rm{M}^n\To \s ^n
$$ 
of $\phi$ is defined as follows: for every $p\in\rm{M}^n$, $G(p)\in \s^n$ is the point at infinity of the unique horosphere $\mathcal{H}_p$ in $\h^{n+1}$ passing through $\phi(p)$ and with the inner unit normal the same as $\eta(p)$ at $\phi(p)$.
\end{definition}

The associated light cone map $\psi$ is strongly related to the Gauss map $G$ of $\phi$. Indeed, the ideal boundary $\s^n$ of $\h^{n+1}$ can be identified with the projective quotient space $\n_+^{n+1} / \r _+$ in such a way that we have 

\begin{equation}\label{psiG}
\psi = e^{\tilde \rho} (1, G),
\end{equation}
where $\tilde\rho$ is the so-called horospherical support function for the hypersurface $\phi$.  Note that horospheres are the unique
hypersurfaces such that, with inward orientation, the associated light cone map  as well as the Gauss map are constant. 
Moreover, if we write $\psi = e^{\tilde\rho}(1,x)$ for a given horosphere, then $x \in \s ^n$ is the point at infinity of the horosphere and $\tilde\rho$ is the signed hyperbolic distance
of the horosphere to the point $\mathcal{O}=(1,0,\ldots ,0) \in \h ^{n+1}\subseteq \r^{1, n+1}$. The intrinsic geometry of a horosphere is Euclidean. Therefore one may introduce a notion of convexity based on horospheres. Namely,

\begin{definition}[ \cite{Sc}]\label{horocon}
Let $\phi: \rm{M}^n \to \h^{n+1}$ be an immersed, oriented hypersurface and let
$\mathcal{H}_p$ denote the horosphere in $\h^{n+1}$ that is tangent to the hypersurface at $\phi(p)$ and whose inward unit normal at $\phi(p)$ agrees with unit normal $\eta(p)$
to the hypersurface $\phi$ at $\phi(p)$. We will say that
$\phi: \rm{M}^n\to\h^{n+1}$ is horospherically convex at $p$ if there exists a neighborhood
$V\subset \rm{M}^n$ of $p$ so that $\phi(V\setminus \{p\})$ does not intersect with $\mathcal{H}_p$. Moreover, the distance function of the hypersurface $\phi: V\to\h^{n+1}$ to the 
horosphere $\mathcal{H}_p$ does not vanish up to the second order at $\phi(p)$ in any direction.
\end{definition}

We have the following characterization of horospherically convex hypersurfaces:

\begin{lemma}[\cite{EGM}]\label{hc}
Let $\phi:\rm{M}^n\To \h^{n+1}$ be an immersed, oriented hypersurface. Then $\phi$ is horospherically convex at $p$ if and only if all principal curvatures of $\phi$ at $p$ are simultaneously $<1$ or $>1$. In particular, $dG$ is invertible at $p$ if $\phi$ is horospherically convex at $p$.
\end{lemma}
To see the second statement, if $\{e_1,\cdots, e_n\}$ denotes an orthonormal basis of principal curvature directions of $\phi$ at $p$ and $\kappa _1, \cdots, \kappa _n$ are the principal
curvatures respectively, i.e.
\begin{equation}\label{kappa-i}
\aligned d\phi(e_i) & = e_i  \\ d\eta(e_i) & = - \kappa_i e_i,
\endaligned
\end{equation}
it is then immediate \cite{EGM} that
\begin{equation}\label{metpsi}
\meta{(d\psi )_p(e_i)}{(d\psi )_p(e_j)}= (1 - \kappa _i)^2 \delta_{ij} = e^{2\tilde\rho}
\meta{(dG)_p(e_i)}{(dG )_p(e_j)}_{\ts ^n}.
\end{equation}
From now on, unless stated otherwise, we will take the orientation on a horospherically convex hypersurface so that all principal curvatures satisfy $\kappa_i  < 1$. 

\begin{remark}
We like to remark here that horospherical convexity (cf. Definition \ref{horocon}) is a weaker notion of convexity for oriented immersed hypersurfaces in hyperbolic Space. To clarify, 
a hypersurface is said to be convex in our terminology if its principal curvatures satisfy $\kappa_i < 0$ in the canonical orientation.
\end{remark}

Now we are ready to introduce the horospherical metric on an immersed horospherically convex hypersurface as follows:

\begin{definition}\label{Def:horom}
Let $\phi: \rm{M}^n\to \h^{n+1}$ be an immersed horospherically convex hypersurface.
Then the Gauss map $G: \rm{M}^n\to \s^n$ is a local diffeomorphism.
We consider the locally conformally flat metric
\begin{equation}\label{horom}
g_h = \psi^*\meta{}{} = e^{2\tilde\rho} G^*g_{\ts ^n} 
\end{equation}
on $\rm{M}^n$ and call it the horospherical metric of the horospherically convex hypersurface $\phi$.
\end{definition}

It is clear that $g_h$ is the induced metric on the immersed hypersurface $\psi: \rm{M}^n\to\n^{n+1}_+$, when $\psi$ is spacelike. Considering 
$\psi:\rm{M}^n\to\n^{n+1}_+\subset\r^{1, n+1}$ as a surface of co-dimension 2 in 
the Minkowski spacetime $\r^{1, n+1}$, we know that $\phi(p)$ and $\eta(p)$ are two unit normal vectors at $\psi(p)$ and the second fundamental form is
$$
II_\psi (e_i, e_j) = (\frac 1{1-\kappa_i}\phi + \frac {\kappa_i}{1 - \kappa_i}\eta)g_h(e_i, e_j).
$$
Hence, the sectional curvature of the metric $g_h$ is
$$
\rm{K}_{g_h}(\frac {e_i}{1 - \kappa_i}, \frac{e_j}{1 - \kappa_j}) = 1 - \frac 1{1 - \kappa_i} - \frac1{1-\kappa_j}
$$
and Schouten tensor is
$$
\rm{Sch}_{g_h} (e_i, e_j) = (\frac 12 - \frac 1{1 - \kappa_i})g_h (e_i, e_j).
$$
When the Gauss map $G:\rm{M}^n\to\s^n$ of a horospherically convex hypersurface $\phi: \rm{M}^n\to \h^{n+1}$ is a diffeomorphism, 
one may push the horospherical metric $g_h$ onto the image $\Omega= G(\rm{M}^n)\subset\s^n$ and consider the conformal metric
$$
\hat g = (G^{-1})^*g_h = e^{2\rho}g_{\ts^n},
$$
where $\rho = \tilde\rho\circ G^{-1}$. For simplicity, we also refer to this conformal metric $\hat g$ as the horospherical metric. On the other hand, given a conformal metric $\hat g = e^{2\rho}g_{\ts^n}$ on a domain $\Omega$ in $\s^n$, one immediately recovers
the light cone map $\psi (x)= e^\rho(1, x): \Omega\to\n^{n+1}_+$. It turns out that one can solve for the map $\phi: \Omega\to\h^{n+1}$ and the unit normal vector $\eta:\Omega\to\s^{n+1}_1$
such that $\phi+\eta = \psi$.

\begin{theorem}[\cite{EGM}]\label{Th:representacion}
Let $\phi: \Omega \subseteq \s ^n\To \h^{n+1}$ be a piece of horospherically convex hypersurface with Gauss map $G(x)=x$. Then 
$\psi = e^\rho(1, x)$ and it holds

\begin{equation}\label{repfor}
\phi = \frac{e^{\rho}}{2}\left( 1+ e^{-2\rho} \left( 1+ |\nabla\rho|^2 \right)\right) (1,x) 
+ e^{-\rho} (0, -x +\nabla\rho).
\end{equation}
Moreover, the eigenvalues $\lambda_i$ of the Schouten tensor of the horospherical metric $\hat g = e^{2\rho}g_{\ts^n}$ and the principal curvatures $\kappa_i$ of $\phi$ are related by

\begin{equation}\label{lambdakappa}
\lambda _i = \frac{1}{2} -\frac{1}{1- \kappa _i} .
\end{equation}
Conversely, given a conformal metric $\hat g= e^{2\rho} g_{\ts^n}$ defined on a domain of the
sphere $\Omega \subseteq \s ^n$ such that the eigenvalues of its Schouten tensor
are all less than $1/2$, the map $\phi$ given
by \eqref{repfor} defines an immersed, horospherically convex hypersurface in $\h ^{n+1}$ whose Gauss map is $G(x)=x$ for $x\in \Omega$ and whose horospherical metric 
is the given metric $\hat g$.
\end{theorem}

To end this subsection, for the convenience of readers, we recall that on a Riemannian manifold $(M ^n , g)$, $n\geq 3$, the Riemann curvature tensor can be decomposed as
$$ {\rm Riem}_g = W_g + {\rm Sch}_g \odot g , $$where $W_g$ is the Weyl tensor,
$\odot$ is the Kulkarni-Nomizu product, and
$$ {\rm Sch}_g := \frac{1}{n-2}\left( {\rm Ric}_g - \frac{S_g}{2(n-1)}g\right) $$is
the Schouten tensor, where ${\rm Ric}_g$ and $S_g$ stand for the Ricci
curvature and scalar curvature of $g$ respectively. The eigenvalues of ${\rm Sch}_g$ are defined as the eigenvalues of the endomorphism $g^{-1}{\rm Sch}_g$. 

\begin{remark}\label{orientation}
To avoid confusion we remind readers that in our convention, for instance, the principal curvatures of a geodesic sphere in hyperbolic space $\h^{n+1}$ with respect to the outward orientation are less than  $-1$. \\

Finally we want to remark that, with the orientation and curvature condition for horospherically convex hypersurfaces here, it is perhaps more appropriate to say concave instead of convex. But, we continue to use these words as used in \cite{EGM} \cite{BEQ}.
\end{remark}


\subsection{Geodesic Defining Functions and Normal Flows}\label{Sect:Flow}

In this section we briefly summarize the work in \cite{BEQ}. We will take a viewpoint that is more reflective of conformal geometry and reinterpret the correspondence, Theorem \ref{Th:representacion},  as the association of conformal metrics and geodesic defining functions. Here one can think of geodesic defining functions as ways of describing foliations 
of hypersurfaces, or level set representations of normal flows. \\

A defining function for a part of the infinity $\Omega\subset\s^n$ of hyperbolic space $\h^{n+1}$ is a smooth function satisfying

\begin{enumerate}
\item  $r > 0$ in $\Omega\times (0, \epsilon_0)\subset\h^{n+1}$;
\item  $r = 0$ on $\Omega\times\{0\}\subset\s^n$; and
\item  $d r \neq 0$ on $\Omega\times\{0\}\subset\s^n$.
\end{enumerate} 
The hyperbolic space $(\h^{n+1}, \gH)$ is conformally compact in the sense that $r^2\gH$ extends to the infinity for any defining function $r$ when considering $\Omega=\s^n$. The metrics $r^2\gH|_{r=0}$ recover the standard conformal class of metrics on $\s^n$ when the defining functions vary. 

\begin{definition}\label{geo-def}
A defining function $r$ is said to be geodesic defining function  if
\begin{equation}\label{normal-flow}
|dr|_{r^2\gH} = 1,
\end{equation}
at least in a neighborhood of the infinity (i.e. $\Omega\times [0, \epsilon_0)$ for some positive number $\epsilon_0$). With geodesic defining function $r$ we have
$$ 
\gH = r^{-2}(dr^2 + g_r),
$$
where $g_r$ is a family of metrics on $\Omega\subset\s^n$. It is easily seen that there is a canonical association between the choice of conformal metric $r^2\gH|_{r=0}$ and the geodesic defining function $r$. 
\end{definition}
The advantage of using geodesic defining functions is evident from the following lemma of Fefferman and Graham \cite{FG2}.

\begin{lemma}[\cite{FG2}] \label{expansion}
Suppose that $g$ is a metric conformal to the standard round metric $g_{\ts^n}$ on a domain $\Omega\subset\s^n$ and that $r$ is the geodesic defining function associated with $g$. Then
$$
\gH = r^{-2}(dr^2 + g_r)
$$
where
\begin{equation}\label{E: Hyp Exp}
g_r = g - r^2 Sch_g + \frac{r^4}{4}Q_g
\end{equation}
and
$$
(Q_g)_{ij} = g^{kl} (Sch_g)_{ik}(Sch_g)_{jl}.
$$ 
\end{lemma}

By the definition of geodesic defining functions above, it is useful to realize that the level surfaces of a geodesic defining function $r$ are the normal flow of the boundary into $\h^{n+1}$ in the conformally compactified metric $r^2\gH$ as well as the normal flow of a horospherically convex hypersurface toward the infinity in $\h^{n+1}$ in the hyperbolic metric $\gH$, which was called parallel flows in \cite{Eps3}. After identifying the level surfaces of a geodesic defining function as horospherically convex hypersurfaces in $\h^{n+1}$, the relation \eqref{lambdakappa} in Theorem \ref{Th:representacion} is a direct consequence of the expansion \eqref{E: Hyp Exp} as observed in \cite{BEQ}. \\

For the convenience of readers we calculate the expansion \eqref{E: Hyp Exp} using Ricatti equations for principal curvatures in hyperbolic space of the normal flow.   
Let $\Omega \subset \s ^n$ be a domain in the sphere and $\phi : \Omega \to \h ^{n+1}$ be an oriented horospherically convex hypersurface so that $G(x)=x$ for all $x \in \Omega \subset  \s ^n$. Let $\set{\phi _t} _{t \in \tr}$ denote the normal flow of $\phi$ in hyperbolic space $\h^{n+1}$, that is, 
\begin{equation}\label{parallel-flow} 
\phi _t (x):= {\rm exp}_{\phi (x)} (t \eta (x))  = \phi(x)\cosh t + \eta(x)\sinh t: \Omega \To\h^{n+1}\subset\r^{1, n+!},
\end{equation}
where ${\rm exp}$ denotes the exponential map for the hyperbolic metric $\gH$.  Due to the Ricatti equations, the principal curvatures $\kappa_i ^t$ of $\phi _t$ are given by 
\begin{equation}\label{kappa-t}
\kappa_i ^t (p) = \frac{\kappa_i (p) - \tanh (t)}{1- \kappa_i (p)\tanh (t)} ,
\end{equation}
and the first fundamental form of $\phi _t$ is given by
\begin{equation}\label{H: Hyp Exp}
I_t (e_i , e_j) =(\cosh (t) - \kappa_i \sinh (t))^2 \delta _{ij} ,
\end{equation}
where $\set{e_1 , \cdots , e_n}$ is an orthonormal basis of principal curvature directions of $\phi$. From here one can easily check that the Gauss maps $G_t$ remain invariant under this flow and the horospherical metric of $\phi _t$ is $g_t := e^{2t}g_h$, where $g_h$ is the horospherical metric of $\phi$. Moreover, the change of variable $r = 2 e^{-t}$ shows that \eqref{H: Hyp Exp} is equivalent to \eqref{E: Hyp Exp}. \\

Conversely, given a conformal metric $\hat g:= e^{2 \rho} g_{\ts^n} $ on $\Omega \subset \s^n$ with Schouten tensor bounded from above, one considers a family of rescaled metric 
$\hat g_t = e^{2t}\hat g$. Choosing $t_0$ large so that $e^{-2t_0}\rm{Sch}_{\hat g} \leq  \frac 12$, it follows from Theorem \ref{Th:representacion} that the foliation of hypersurfaces 

\begin{equation}\label{phi-t}
\phi_t = \frac{e^{\rho+t}}{2}\left( 1+ e^{-2\rho-2t} \left( 1+ |\nabla\rho|^2 \right)\right) (1,x) 
+ e^{-\rho-t} (0, -x +\nabla\rho):\Omega\To\h^{n+1}
\end{equation}
for $t > t_0$ consists of immersed, horospherically convex hypersurfaces with Gauss map $G_t(x) = x:\Omega\to\s^n$ the identity. 


\section{Global Theory}\label{Sect: global}

From the previous section, we know that, for a piece of horospherically convex hypersurface in hyperbolic space $\h^{n+1}$, the Gauss map induces a canonical conformal metric on the infinity $\s^n$ locally. Conversely, given a conformal metric on a domain of $\s^n$, there is an immersed, horospherically convex hypersurface in hyperbolic space $\h^{n+1}$ whose horospherical metric is the given metric up to a rescale. In this section we establish a global correspondence between properly immersed, complete, horospherically convex hypersurfaces and complete conformal metrics on domains of $\s^n$.  Given a complete, properly immersed, horospherically convex hypersurface $\phi: \rm{M}^n\To\h^{n+1}$, the issues that concern us are the following:  

\begin{itemize}
\item
When is the horospherical metric $g_h$ complete? 
\item
When is its Gauss map injective?  
\item
When does the boundary at infinity of the hypersurface coincide with the boundary of the Gauss map image?  
\end{itemize}
In the other direction, given a complete conformal metric $\hat g$ on a domain $\Omega$ of the infinity $\s^n$ with Schouten tensor bounded from above, we are concerned with the following
issues:

\begin{itemize}
\item
When does it correspond to a complete, immersed, horospherically convex hypersurface? 
\item
When is the corresponding hypersurface proper?
\item
When does the boundary of the domain coincide with the boundary at infinity of the hypersurface?
\end{itemize}
A final, yet most important question is: when are the leaves of the normal flow given in \eqref{parallel-flow} or \eqref{phi-t} 
eventually embedded? Equivalently, one may ask when there is a geodesic defining function associated with
a given complete conformal metric on a domain $\Omega\subset\s^n$ defined for a positive distance uniformly in the domain $\Omega$. 

\subsection{Uniform Convexity vs Bounded Curvature}\label{Sect: uniform-bounded}

We are able to make a satisfactory correspondence if we restrict ourselves to the cases where hypersurfaces are uniformly horospherically convex or equivalently 
the conformal metrics are of bounded curvature. Let us start with the definition of uniformly horospherically convex.

\begin{definition}\label{uniform-cv} Let $\phi: \rm{M}^n\to \h^{n+1}$ be an immersed, oriented hypersurface. We say that $\phi$ is uniformly horospherically convex if there is a number 
$\kappa_0<1$ such that all the principal curvatures $\kappa_i$ at all points in $\rm{M}^n$ are less than or equal to $\kappa_0$.
\end{definition}

Hence, in the light of \eqref{lambdakappa}, one can easily see that, for a conformal metric $\hat g = e^{2\rho}g_{\ts^n}$ on a domain $\Omega\subset\s^n$ with Schouten tensor
bounded from above, the corresponding hypersurface $\phi_t$ given in \eqref{phi-t} is an immersed, uniformly horospherically convex hypersurface for $t$ large enough if and only if
the Schouten tensor of $\hat g$ is also bounded from below. On the other hand, when the conformal metric is of bounded curvature, the corresponding 
hypersurfaces $\phi_t$ given in \eqref{phi-t} are immersed and uniformly horospherically convex with bounded principal curvatures for $t$ large enough. Based on the above observation we make the following definition.

\begin{definition}\label{admissible-realizable} An oriented hypersurface $\phi: \rm{M}^n\to\h^{n+1}$ is said to be admissible if it is properly immersed, complete, uniformly horospherically convex with injective Gauss map $G:\rm{M}^n\to\s^n$. Meanwhile, a complete conformal metric $\hat g= e^{2\rho}g_{\ts^n}$ on a domain $\Omega\subset\s^n$ is said to be realizable if it is of 
bounded curvature.
\end{definition}

When we start with a properly immersed, complete, horospherically convex hypersurface $\phi: \rm{M}^n\to\h^{n+1}$ with injective Gauss map $G:\rm{M}^n\to\s^n$, from Theorem \ref{Th:representacion}, we know $\phi$ induces a conformal metric $\hat g$ on the image of the Gauss map $\Omega = G(\rm{M}^n)\subset\s^n$ with Schouten tensor bounded from 
above by one half. Then, the question to ask is if the conformal metric $\hat g$ is complete? One can easily construct an example to show that the answer in general is negative. We will present 
a properly immersed, complete, horospherically convex hypersurface whose horospherical metric is not complete at the end of this subsection. On the other hand, when the hypersurface is uniformly horospherically convex, the completeness of the horospherical metric is a simple consequence of  \eqref{metpsi}.

\begin{lemma}\label{uniform-complete} Suppose that $\phi: \rm{M}^n\to\h^{n+1}$ is a complete, immersed, uniformly horospherically convex hypersurface. Then the horospherical metric $g_h$ is
complete on $\rm{M}^n$.
\end{lemma}

When we start with a complete conformal metric $\hat g = e^{2\rho}g_{\ts^n}$ on a domain $\Omega\subset \s^n$ with Schouten tensor bounded from above, from Theorem \ref{Th:representacion},  we know that for $t$ large enough the hypersurface $\phi_t$ given by  \eqref{phi-t} is immersed and horospherically convex . Then a natural question to ask is if the hypersurface $\phi_t$ is complete and proper.  One again easily observes 

\begin{lemma}\label{proper-complete} Suppose that $\hat g = e^{2\rho}g_{\ts^n}$ is a complete conformal metric on a domain $\Omega\subset\s^n$ with Schouten tensor bounded 
from above and that $\phi_t: \Omega\to\h^{n+1}$ given by \eqref{phi-t} is immersed.  In addition we assume that 
$$
\beta (x) := e^{2\rho (x)} + |\nabla\rho|^2 (x) \To +\infty \quad\text{as $x\to\partial\Omega$}.
$$ 
Then $\phi_t$ is a properly immersed, complete, horospherically convex hypersurface for $t$ large enough.
\end{lemma}

\proof Here we shall use the Poincar\'{e} ball model of $\h^{n+1}$. We like to use stereographic projection in Minkowski spacetime to realize the coordinate change between the two models. Namely,
$$ 
\begin{matrix} 
 & \h^{n+1} \subset \r^{1, n+1}  & \To & \b^{n+1}  \subset \r^{n+1} =\{\bar x\in\r^{1, n+1}: x_0=0\} \\  &  & \tau  & \\ & (x_0, x_1 , \cdots , x_{n+1}) & \To  
 & \dfrac{1}{1+ x_0} (x_1 , \cdots ,  x_{n+1})  \end{matrix}.
$$
Hence, omitting the variable $t$ for simplicity, we have, 
$$ 
\tau\circ\phi = \dfrac{e^{2\rho} + |\nabla\rho|^2 - 1}{e^{2\rho} +  2e^\rho + |\nabla\rho|^2 + 1}\left( x + Y(x) \right): \Omega\To \b^{n+1} 
$$ 
with 
$$ 
Y(x)= \dfrac{2}{e^{2\rho} +|\nabla\rho|^2 -1} \nabla\rho .
$$
Now it is easily seen that if $\beta (x) \to + \infty$, then
\begin{align*}
\left(\dfrac{e^{2\rho} + |\nabla\rho|^2 -1}{e^{2\rho} +2e^\rho + |\nabla\rho|^2 +1  }\right) (x) &\To  1 \\
\intertext{and}
\left(\dfrac{2}{e^{2\rho} + |\nabla\rho|^2 - 1} \nabla\rho \right) (x) &\To  {\bf 0}.
\end{align*}
Therefore, if $\beta (x) \to + \infty$ as $x\to x_0\in\partial\Omega$, it then follows that 
\begin{equation}\label{continuity}
\tau\circ\phi(x) \to x_0
\end{equation}
as desired.
\endproof

One important side product of the proof of Lemma \ref{proper-complete} is the following:

\begin{corollary}\label{bdy-regularity} Suppose that $\hat g = e^{2\rho}g_{\ts^n}$ is a complete conformal metric on a domain $\Omega\subset\s^n$ with Schouten tensor bounded 
from above and that $\phi_t: \Omega\to\h^{n+1}$ is given in \eqref{phi-t}.  In addition we assume that 
$$
\beta (x) := e^{2\rho (x)} + |\nabla\rho|^2 (x) \To +\infty \quad\text{as $x\to\partial\Omega$}.
$$ 
Then $\phi_t$ is a properly immersed, complete, horospherically convex hypersurface, and
$$
\partial_\infty\phi_t (\rm{M}^n) = \partial\Omega,
$$  
for $t$ large enough.
\end{corollary}
 
One may refer to Definition \ref{infinity} for the boundary at infinity $\partial_\infty \phi(\rm{M}^n)$ of a hypersurface $\phi$ in hyperbolic space $\h^{n+1}$. It seems to us that it is a rather subtle issue to determine when $\beta(x)\to+\infty$ as $x\to\partial\Omega$ if one only assumes the metric $\hat g$ to be compete and the Schouten tensor to be bounded from above. We settle the issue by using Proposition 8.1 in \cite{CHY}, where it is shown that the conformal factor $\rho\to +\infty$ as $x\to\partial\Omega$ if the scalar
curvature is bounded from below. Notice that in our context, since we always assume the Schouten tensor is bounded from above, the fact that the scalar curvature is bounded 
from below implies the curvature of the conformal metric $\hat g$ is bounded. 
 
\begin{proposition}\label{chy} Suppose that $\hat g = e^{2\rho}g_{\ts^n}$ is a complete conformal metric on a domain $\Omega\subset\s^n$ with Schouten tensor bounded 
and that $\phi_t: \Omega\to\h^{n+1}$ is given in \eqref{phi-t}. Then $\phi_t$ is a properly immersed, complete, uniformly horospherically convex hypersurface with uniformly 
bounded principal curvature, and
$$
\partial_\infty\phi_t (\rm{M}^n) = \partial\Omega,
$$
for $t$ large enough. 
 \end{proposition}
 
Again, one may refer to Definition \ref{infinity} for the boundary at infinity $\partial_\infty \phi(\rm{M}^n)$ of a hypersurface $\phi$ in hyperbolic space $\h^{n+1}$. To summarize we have the following main result in this subsection for the global correspondence.

\begin{theorem}\label{uniform-bounded} Suppose that $\phi: \rm{M}^n\to\h^{n+1}$ is an admissible hypersurface with the hyperbolic
Gauss map $G:\rm{M}^n\to\Omega\subset\s^n$. Then it induces a realizable metric on the domain $\Omega$. Moreover $\partial_\infty\phi (\rm{M}^n) 
= \partial\Omega$. \\

On the other hand, suppose that $e^{2\rho}g_{\ts^n}$ is a realizable metric on a domain $\Omega\subset\s^n$. Then $\phi_t$ given in \eqref{phi-t} is an admissible hypersurface
with bounded principal curvature and $\partial_\infty\phi_t(\rm{M}^n) = \partial\Omega$, for $t$ large enough.
\end{theorem}

The above result provides a back-and-forth relationship between complete conformal metrics on domains of the sphere and horospherically convex hypersurfaces in $\h^{n+1}$ with prescribed boundary at infinity. This allows to relate the results of \cite{LN} and \cite{MP} for singular solutions for conformal metrics on the sphere with those of, among others, \cite{S, GS} 
for hypersurfaces in $\h^{n+1}$ with prescribed boundary at infinity.  \\

Before we end this subsection we would like to present an easy example to show that one in general does not get the completeness of horospherical metric. Let us consider $\Omega = \s^{n-1}\times (-1, 1) \subset \s^{n-1} \times (-\frac \pi 2,\frac \pi 2) = \s^n\setminus \{{\mathbf S}, {\mathbf N}\} \subset \s^n$. In this parameterzation,  the standard round metric is given as
$$
g_{\ts^n} = ds^2 + \cos^2 s g_{\ts^{n-1}}
$$
and the Christoffel symbols are 
$$
\Gamma^s_{ss}= \Gamma^s_{si} = 0 \text{ and } \Gamma^s_{ij} = \tan s (g_{\ts^n})_{ij}
$$
for $i, j = 2, 3, \cdots, n$. Let 
$$
\rho (\theta , s) =\rho (s) = - \frac 12 \log ( 1- s^2)
$$ 
and $\hat g = e^{2 \rho} g_{\ts^n}$ be the conformal metric on $\Omega$. If we consider the meridian $\gamma : (0, 1) \to \Omega $ given by $\gamma (s) = (\theta _0 , s) $ 
where $\theta _0 \in \s ^{n-1}$ is fixed, then we easily see that
$$ 
\int _\gamma  e ^{\rho}  \, dv_{g_{\ts^n}} = \int_0^1 \frac 1{\sqrt{1 - s^2}} ds  < +\infty ,
$$
which implies that $\hat g$ is not complete in $\Omega$. On the other hand,  we recall
$$
{\rm Sch}[\hat g]_{ik} = {\rm Sch}[g_{\ts^n}]_{ik} - \rho_{i,k} + \rho_i\rho_k - \frac 12 |\nabla\rho|^2 (g_{\ts^n})_{ik}
$$
and calculate 
$$
\rho_s = \frac s{1 - s^2}
$$
and the only nonzero terms for the Hessian are 
$$
\rho_{s,s} =  \frac {1 +  s^2}{(1-s^2)^2},\quad  \rho_{i,j} = -  \frac {s\tan s}{ 1 -s^2}(g_{\ts^n})_{ij}.
$$
Hence we notice
$$
- \rho_{s, s} + \rho_s^2 - \frac 12 \rho_s^2 = - \frac {1 +s^2}{(1-s^2)^2} + \frac 12 \frac {s^2}{(1 -s^2)^2} = - \frac {1 +\frac 12s^2}{(1 -s^2)^2} < 0
$$
and
$$
 - \rho_{i,j} + \rho_i\rho_j - \frac 12 |\nabla\rho|^2 (g_{\ts^n})_{ij}  = (\frac {s\tan s}{1 - s^2} - \frac {s^2}{2(1-s^2)^2})  (g_{\ts^n})_{ij} \leq C(g_{\s^n})_{ij}
 $$
for some $C>0$, $i, j = 2, 3, \cdots, n$, and $s\in (-1, 1)$. Therefore  we consider the immersed, horospherically convex hypersurface $\phi_t$ given by \eqref{phi-t} corresponding to $(\Omega , \hat g)$
for $t$ sufficiently large.  Since $\rho \to +\infty$ as $s$ approaches $1$, from Lemma \ref{proper-complete}, we know that $\phi_t$ is proper and complete.
 We remark here that in  fact ${\rm Sch}[\hat g]$ is not bounded from below in this example, which implies that the hypersurface $\phi_t$ is 
not uniformly horospherically convex.

\subsection{Injectivity of Hyperbolic Gauss Maps}\label{Sect:injectivity}

We next describe an explicit example to show that indeed the Gauss map of a noncompact, complete, properly immersed oriented horospherically convex hypersurface may 
not be injective. The essential reasons are that one can have a convex, self-intersecting, closed curve in $\h^2$ and that higher dimensional hyperbolic space $\h^{n+1}$ is a foliation of totally geodesic $\h^2$ via translation isometries. \\

Let $r,R: \r \to \r $ be smooth $4\pi -$periodic functions defined by
$$ 
r(u):= \sin(\frac{u}{2}) \cos (u) , \, \, \, R(u):= \cos (\frac{u}{2})- \frac{1}{3} \cos(\frac{3u}{2}),
$$
and let $ \alpha (u) : \r \to \h ^2\subset \r^{1, 2}$ be given by
$$ 
\alpha (u) = (\cosh (r(u)) \cosh (R(u)), \sinh (r(u))\cosh (R(u)), \sinh (R(u))).
$$
Then $\alpha$ is non-embedded and has nonnegative curvature. Actually, in the geodesic coordinate, its profile is as depicted:

\begin{center}
\DeclareGraphicsExtensions{jpg}
\includegraphics[scale=0.8]{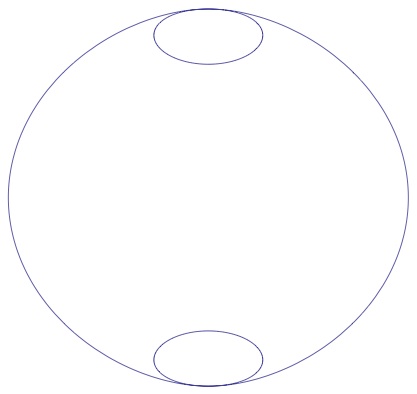}\\
\mbox{{\small $\alpha (u) := (\sin(\frac{u}{2}) \cos (u), \cos (\frac{u}{2})- \frac{1}{3} \cos(\frac{3u}{2}))$}}
\end{center}
So the desired hypersurface is generated from the above immersed convex closed curve in a totally geodesic surface $\h^2$ by $(n-1)$-families of translation isometries along geodesics orthogonal to the totally geodesic surface $\h^2$ in $\h^{n+1}$. The resulting hypersurface is a properly immersed convex hypersurface $\phi: \r^{n-1}\times \s^1\to \h^{n+1}$ where by construction the principal curvatures of the hypersurface are all zero except one is positive. Hence, the scalar curvature of the horospherical metric of such immersed convex hypersurface is strictly negative when $n\geq 3$.
Also, the boundary at infinity of this hypersurface is a $\s^{n-2}$. In other words, the image of the Gauss map is $\s^n\setminus\s^{n-2} \simeq \r^{n-1}\times\s^1$, which is not simply connected.
Considering the normal vector along the profile curve one sees that the Gauss map is a three-sheet covering map. \\

In general, it is a rather difficult issue to determine when the Gauss map is injective. On the other hand, the Gauss map of an immersed, horospherically convex hypersurface in hyperbolic space is a development map from the parameter space ${\rm M}^n$ equipped with the horospherical metric into the sphere.
Hence, due to Kulkarni and Pinkall \cite{K-P}, we have the following:

\begin{lemma}\label{covering} 
Suppose that $\phi: {\rm M}^n\to\h^{n+1}$ is an immersed, horospherically convex hypersurface and that the horospherical metric $g_h$ is complete on ${\rm M}^n$. Then the Gauss map is a covering onto its image in the sphere. Hence, the Gauss map is injective if its image in the sphere is simply connected.
\end{lemma}

In this subsection we will show that the Gauss map of a properly immersed, complete,  horospherically convex hypersurface is injective when it is regular at infinity 
(cf. Definition \ref{regular-end}) and the boundary at infinity (cf. Definition \ref{infinity}) is small, which will be made precise below. We will also show that the injectivity of the Gauss map follows under certain curvature conditions on the hypersurface, which is a straightforward consequence of the celebrated injectivity of development maps of Schoen and Yau \cite{SY, SY1}.  \\
 
In the light of Lemma \ref{covering} the Gauss map is injective when its image in the sphere is simply connected. 
A good way to study images of Gauss maps is to consider the boundaries at infinity
of hypersurfaces. Let us first define the boundary at infinity of a noncompact hypersurface in $\h^{n+1}$. 

\begin{definition}\label{infinity}
Suppose that $\phi: {\rm M}^n\to\h^{n+1}$ is a properly immersed hypersurface. We define the boundary at infinity 
$\partial_\infty\phi({\rm M}^n)$ to be the collection of points $x\in\s^n$ such that 
there is a sequence $x_n$ on the hypersurface in the Poincar\'{e} ball $\b^{n+1}$ model of hyperbolic space that converges to $x$ in $\overline{\b^{n+1}}$ in Euclidean topology.
\end{definition}

In general, the end behaviors of properly immersed, complete, horospherically convex hypersurfaces may be very wild. The following regularity of Gauss maps at infinity seems to be
a very efficient and geometric way to restrict the behavior of the end and in many ways excludes the persistent sharp turns of a surface approaching the boundary at infinity.

\begin{definition}\label{regular-end}
Suppose that $\phi: {\rm M}^n\to \h^{n+1}$ is a properly immersed hypersurface. The Gauss map is said to be regular at infinity if,  for each $p\in \partial_\infty\phi({\rm M}^n)\subset \s^n$,
$$
\lim_{i\to \infty} G(q_i) = p
$$
for $q_i\in {\rm M}^n$, $\phi (q_i)\to p$. 
\end{definition}

As a consequence of the regularity of the Gauss map at infinity,  we have the following:

\begin{lemma}\label{end-boundary} Suppose that $\phi: {\rm M}^n\to \h^{n+1}$ is a properly immersed, complete, horospherically convex hypersurface and that the Gauss map $G: {\rm M}^n\to\s^n$ is regular at infinity. Then
\begin{equation}\label{reg-bdy}
\partial G({\rm M}^n) \subset \partial_\infty \phi({\rm M}^n).
\end{equation}
\end{lemma}

\proof Let $p\notin\partial_\infty\phi({\rm M}^n)$. We would like to show that $p\notin\partial G({\rm M}^n)$. Otherwise, $p\in \partial G({\rm M}^n)$, which means that $p\notin G({\rm M}^n)$ and there is a sequence $p_i\in G({\rm M}^n)$ such that $p_i\to p$ in $\s^n$. Let $q_i\in {\rm M}^n$ such that $G(q_i) = p_i$. At least for a subsequence, we may assume $\phi(q_i)$ converges to $x \in \bar{\mathbb B}^{n+1}$. By the completeness of the hypersurface, if $x \in \b^{n+1}$, then $p= G(q)$ for some $q\in {\rm M}^n$ and $\phi(q) = x$, which contradicts the fact that
$p\notin G({\rm M}^n)$. On the other hand, if $x\in \partial\b^{n+1}$,  one may conclude that $x \in\partial_\infty\phi({\rm M}^n)\subset\s^n$ by the regularity of the Gauss map, which
contradicts the fact that $p\notin\partial_\infty\phi({\rm M}^n)$. 
\endproof

We also observe the following:

\begin{proposition}\label{finite-covering} Suppose that $\phi: {\rm M}^n\to \h^{n+1}$ is a properly immersed, complete, horospherically convex hypersurface with complete horospherical metric and  
that the Gauss map is regular at infinity. Then either the Gauss map is a finite covering or
$$
G({\rm M}^n)\subset \partial_\infty\phi({\rm M}^n).
$$
\end{proposition}

\proof For any $p\in G({\rm M}^n)$, we consider the preimage $G^{-1}(p)\subset{\rm M}^n$ and the set $P = \{\phi(q): q\in G^{-1}(p)\} \subset \b^{n+1}$ of points on the surface. 
First we show that no limit point of $P$ is inside $\b^{n+1}$. Otherwise, suppose that $x\in \b^{n+1}$ is a limit point of $P$. 
Then $x\in \phi({\rm M}^n)$ due to the completeness of the surface. By the properness of the immersion $\phi$ we may conclude that $G^{-1}(p)$ has a limit point in ${\rm M}^n$, which contradicts 
the fact that the Gauss map is a local diffeomorphism.  \\

On the other hand, since $\phi$ is proper, when $G^{-1}(p)$ is infinite so is the set $P$. In this case $P$ can only have limit points in the boundary at infinity $\partial_\infty\phi({\rm M}^n)$. Therefore, $p\in \partial_\infty\phi({\rm M}^n)$ due to the regularity of the Gauss map at infinity. The conclusion of this proposition then follows from Lemma \ref{covering}. 
\endproof

Proposition \ref{finite-covering} tells us that the Gauss map is a finite covering when the Gauss map is regular at infinity and the boundary at infinity $\partial_\infty\phi({\rm M}^n)$ of the surface has no interior points. We know that a subset in $\s^n$ has no interior point if, for example, it is of Hausdorff dimension less than $n$. In fact, when the boundary at infinity $\partial_\infty\phi({\rm M}^n)$ of a surface is of Hausdorff dimension less than $n-2$, it turns out that the Gauss map has to be injective.

\begin{theorem}\label{small-infinity}
Suppose that $\phi: {\rm M}^n\to \h^{n+1}$ is a properly immersed, complete,  horospherically convex hypersurface with complete horospherical metric and that its Gauss map is regular at infinity.  Then the Gauss map is injective and
$$
\partial G({\rm M}^n) = \partial_\infty \phi({\rm M}^n),
$$
provided that $\partial_\infty\phi({\rm M}^n)\subset \s^n$ is small in the sense that its Hausdorff dimension is less than $n-2$.
\end{theorem}

\begin{proof} By the above Lemma \ref{end-boundary}, we know that
$$
\partial G({\rm M}^n) \subset \partial_\infty\phi({\rm M}^n)
$$
is small in the sense that its Hausdorff dimension is less than $n-2$. Then $G({\rm M}^n)$ is connected and simply connected in $\s^n$. Because, any loop in $G({\rm M}^n)$ can be deformed into a point in $\s^n$ without leaving $G({\rm M}^n)$, when $\s^n\setminus G({\rm M}^n)$ is of codimension bigger than $2$ in $\s^n$. Notice that $\s^n\setminus G({\rm M}^n)
=\partial G({\rm M}^n)$ when $\partial G({\rm M}^n)$ is of Hausdorff dimension less than $n-1$. \\

Thus, in the light of Lemma \ref{covering}, the Gauss map is injective. This implies that no point in $\partial_\infty\phi(M^n)$ can be in the image $G({\rm M}^n)$ of the Gauss map, which implies
$$
\partial G(M^n) = \partial_\infty\phi(M^n).
$$
\end{proof}

As noted earlier, the Gauss map is a development map from a locally conformally flat manifold $(M^n, g_h)$ into $\s^n$. Therefore, we may apply the celebrated result on the injectivity of the developing map in \cite{SY, SY1}.

\begin{theorem}\label{inj}
Suppose that $\phi: M^n\to \h^{n+1}$ is an immersed, complete, horospherically convex hypersurface and suppose that 
\begin{equation}\label{scalar}
\sum_{i=1}^n\frac 2{1- \kappa_i} \leq n,
\end{equation}
where $\kappa_i$ are the principal curvatures of $\phi$ in $\h^{n+1}$. Then the Gauss map is injective. Hence, the hypersurface $\phi$ is admissible and
$$
\partial G({\rm M}^n) = \partial_\infty\phi({\rm M}^n).
$$
\end{theorem}
 
\begin{proof} This turns out to be a rather straightforward consequence of Theorem 3.5 on page 262 in \cite{SY1}. First, the assumption \eqref{scalar} implies that the hypersurface is 
in fact uniformly horospherically convex. Hence, the horospherical metric $g_h$ is complete in the light of \eqref{metpsi}. Secondly, due to the explicit relation \eqref{lambdakappa} in 
Theorem \ref{Th:representacion}, the assumption \eqref{scalar} implies that the scalar curvature of the horospherical metric $g_h$ is nonnegative. Thus, by Theorem 3.5 on page 262 in the book \cite{SY1}, the Gauss map $G$ of $\phi$ as a development map is injective. The remaining claim then follows from the fact that the surface $\phi$ is now known to be admissible. 
\end{proof}

For the convenience of readers we provide Theorem 3.5 on page 262 of \cite{SY1} in the following:

\begin{theorem}[\cite{SY1}]
Let $(M^n, g)$ be a complete Riemannian manifold with nonnegative scalar curvature. Suppose that $\Phi: M^n\to \s^n$ is a conformal map. Then $\Phi$ is injective and $\partial\Phi(M^n)\subset\s^n$ has zero Newton capacity.
\end{theorem}


\subsection{Embeddedness}\label{Sect:embedded}

An important issue in the theory of hypersurfaces is to know when an immersed hypersurface is in fact embedded. To use the convexity to gain embeddedness is a classic idea traced 
back to Hadamard \cite{Had, St}. Here we will combine the ideas from \cite{Cu, EGR} and the connection between normal flows and geodesic defining functions observed in \cite{BEQ}
to obtain some embedding theorems, based on the embedding theorem in \cite{do-warner}. First we state some extremal cases where the hypersurfaces $\phi_t$  in \eqref{phi-t} are embedded from known results in \cite{Al, Cu}.

\begin{proposition}\label{Lem:embedded}
Let $\hat g = e^{2\rho}g_{\ts^n}$ be a realizable metric on a domain $\Omega\subset\s^n$ . 

\begin{itemize}
\item[a)] If the Schouten tensor of the conformal metric $g$ is nonnegative, then $\Omega $ is either $\s^n$ or $\s^n \setminus \set{point}$.  In the first case the corresponding
hypersurface $\phi_t$ given in \eqref{phi-t} is an embedded ovaloid when $t$ is large . In the later case the corresponding hypersurface  $\phi_t$ is a horosphere with the inward orientation for  each $t$. 

\item[b)] On the other hand, if the Schouten tensor is nonpositive, then $\Omega$ is  homeomorphic to $\r^n$ and the corresponding horospherically convex hypersurface $\phi_t$ given in \eqref{phi-t} is properly embedded for all $t$. 
\end{itemize}
\end{proposition}

\begin{proof}

a) By Theorem \ref{uniform-bounded}, we know that $\phi_t$ as in \eqref{phi-t} is a properly immersed, complete, uniformly horospherically convex hypersurface when $t$ is large enough. Moreover, 
the nonnegativity of the Schouten tensor implies that all principal curvatures of the surface $\phi_t$ are less than or equal to $-1$. 
Thus, from \cite{Cu},  the surface is either an embedded $n$-sphere or a horosphere.  \\

b) As in the above, we construct properly immersed, complete, uniformly horospherically convex hypersurfaces $\phi_t$ via Theorem \ref{uniform-bounded}. 
This time the nonpositivity of the Schouten tensor of a realizable metric implies that all principal curvatures of the surface $\phi_t$ are between $-1$ and $1$. Thus, from \cite{Al}, the surface is properly embedded and homeomorphic to $\r^n$.  
\end{proof}

In general, one simply cannot expect all admissible hypersurfaces are embedded, as it was pointed out in \cite{EGM}, even a horospherical ovaloid may not be embedded. But what we can hope is that every admissible hypersurface can be unfolded along the normal flow into an embedded one, which is shown to be the case for a horospherical ovaloid in Corollary \ref{cpt-embedding}. We recall that the geodesic defining function $r$ and its level surfaces give rise to both the normal flow in the hyperbolic metric $\gH$ (called parallel flows in \cite{Eps3}) and the normal flow in the compactified metric $r^2\gH$. It is worth mentioning that the geodesic defining function is not well defined when the surfaces are no longer embedded, while the normal flow of immersed surfaces is still well defined. Therefore, the embeddedness of a hypersurface is equivalent to the existence of a geodesic defining function, which is equivalent to solving the noncharacteristic first order partial differential equation \eqref{normal-flow}. Interestingly, solving \eqref{normal-flow} by the  characteristic method is equivalent to solving the normal flow in the compactified metric. It then becomes a standard geometric question of how far one can push a totally geodesic hypersurface along the normal flow in a Riemannian manifold without any focal points, to which the standard Riemannian comparison theorems apply (cf. \cite{Cheeger-Ebin}). This reconfirms that given a conformal metric with Schouten tensor bounded from above, the hypersurface $\phi_t$ given by \eqref{phi-t} is an immersion when $t$ is large enough (cf. \cite{EGM}).
 
\begin{theorem}\label{Th:Immersion}
Suppose that $\hat g= e^{2\rho}g_{\ts^n}$ is a conformal metric on $\Omega\subset \s^n$ with Schouten tensor bounded from above. Then $\phi_t$ given in \eqref{phi-t} is an immersion for $t$ large enough.
\end{theorem}

\begin{proof} Given a point $x\in \Omega$, there exists an open neighborhood $U$ of $x$ inside $\Omega$ on which the geodesic defining function associated with the given metric 
$\hat g=e^{2\rho}g_{\ts^n}$ on $\Omega$ is defined and reaches out for a positive number $\epsilon$.
The key point is to show that $\epsilon \geq \epsilon_0$ for some positive number $\epsilon_0$ which is independent of where $x$ is in $\Omega$. In the compactified metric
$r^2\gH$ the infinity $U$ becomes a totally geodesic boundary and how far the geodesic defining function can be defined is determined by how far the boundary can be pushed in along the normal flow without encountering any focal points. In other words, in the metric $r^2\gH$, we would like to know how far the geodesics from points in $U$ in the directions normal to the boundary $U$ can be extended without collisions locally. \\

Let us calculate the sectional curvatures for the metric $r^2\gH$ in the plane containing the direction normal to $U$. To do so, we set a normal coordinate $x$ with respect to the metric 
$\hat g$ at a point $x_0$ in $U$. We may assume that the Schouten tensor of the conformal metric $\hat g$ is in diagonal form under the chosen coordinates at $x_0$. Hence, in 
the coordinate $(r, x)$ for $\bar g = r^2\gH$,
$$
\bar R_{irir}  = \frac 12(-\partial_r\partial_r \bar g_{ii} - \partial_i\partial_i \bar g_{rr} + \partial_i\partial_r \bar g_{ ir} + \partial_r\partial_i\bar g_{ir})  - \bar g^{\alpha\beta}
(\overline{[ii, \alpha]} \, \overline{[rr, \beta]} - \overline{[ir, \alpha]} \,  \overline{[ir, \beta]}),
$$
where the Christoffel symbols of second kind are given by
$$
\overline{[\alpha\beta, \gamma] }= \frac 12(\partial_\beta \bar g_{\alpha\gamma} + \partial_\alpha\bar g_{\beta\gamma} - \partial_\gamma\bar g_{\alpha\beta})
$$
What is good here is that we only need to take second order derivatives for $\bar g_{rr}$ with respect to $x$ variables. It is helpful at this point to recall from Lemma \ref{expansion} that
$$
\bar g = r^2\gH = dr^2 + \hat g - r^2 Sch_{\hat g} + \frac{r^4}{4}Q_{\hat g}. 
$$
Hence,
\begin{equation}\label{section-curvature}
\aligned
\bar R_{irir}  & = -\frac 12 \partial_r\partial_r \bar g_{ii} + \bar g^{ii}[ir, i][ir,i]\\
& = \lambda_i  - \frac 32 r^2 \lambda_i^2 + (1 - r^2 \lambda_i + \frac 14 r^4 \lambda_i^2)^{-1}( r\lambda_i - \frac 12 r^3\lambda_i^2)^2 \\
& = \lambda_i  - \frac 12 r^2 \lambda_i^2.
\endaligned
\end{equation}
Now one may apply the second Rauch comparison theorem of Berger, Theorem 1.29 on page 30 of the book \cite{Cheeger-Ebin} to conclude that, for any given neighborhood $V$ of $x_0$ 
such that $\bar V\subset U$, there is a positive number $\epsilon_0$ such that the geodesic defining function reaches beyond $\epsilon_0$ from any point in $V$, since $Sch_{\hat g}$ is assumed 
to bounded from above. 
\end{proof}

Consequently, even though a horospherical ovaloid may not be embedded (cf. \cite{EGM}), it can be expanded along the normal flow into an embedded one.

\begin{corollary}\label{cpt-embedding} Suppose that $\phi: {\rm M}^n\to\h^{n+1}, n \geq 2,$ is a connected, compact, immersed, horospherically convex hypersurface. Then the leaves $\phi_t$ in \eqref{parallel-flow} in the normal flow are embedded spheres when $t$ is large enough.
\end{corollary}

This should be compared with the Hadamard type theorem established by Do Carmo and Warner in Section 5 in \cite{do-warner} (cf. \cite{Had, St}). For the convenience of readers we state their result as follows:

\begin{theorem}[\cite{do-warner}]\label{do-warner} Suppose that $\phi: {\rm M}^n\to\h^{n+1}, n \geq 2,$ is a connected, compact, immersed hypersurface with all principal curvature nonnegative. 
Then $\phi$ is an embedded ovaloid.
\end{theorem}

It turns out that Theorem \ref{do-warner} is one of the important key ingredients in our approach to establish the embeddedness of leaves in the normal flow from a noncompact 
admissible hypersurface. Another key ingredient is also a consequence of Theorem \ref{Th:Immersion}.

\begin{lemma}\label{cpt-part} Suppose that $\phi: \Omega\to\h^{n+1}, n\geq 2,$ is an immersed, horospherically convex hypersurface with Gauss map $G(x) = x: \Omega\to\s^n$. Then, for any
compact subset $K\subset\Omega$, the hypersurfaces $\phi_t: K\to\h^{n+1}$ given in \eqref{parallel-flow} are embedded when $t$ is sufficiently large.
\end{lemma}

In order to apply Theorem \ref{do-warner} we use the following: 

\begin{lemma}\label{strong-convex} Suppose that $\phi: {\rm M}^n\to\h^{n+1}$ is an admissible hypersurface. Then 
\begin{equation}\label{eq:strong-convex}
\kappa_i^t = \frac {\kappa_i - \tanh t}{1 -\kappa_i\tanh t}  = -\coth t + \frac {\coth^2 t - 1}{\coth t - \kappa_i} \To -1 \text{ as $t\to\infty$},
\end{equation}
where $\kappa_i^t$ are the principal curvatures for the hypersurface $\phi_t$ given by \eqref{phi-t}.
\end{lemma}

\proof This is simply because the principal curvatures $\kappa_i \leq \kappa_0$ for some $\kappa_0< 1$ from the uniformly horospherical convexity. 
\endproof

\begin{theorem}\label{Th:Correspondence}
Suppose that $\phi: \Omega\to \h^{n+1}$ is an admissible hypersurface and that the hypersurfaces $\phi_t$ given by \eqref{parallel-flow} is the normal flow from $\phi$. In addition, we assume 
that the boundary $\partial\Omega$ at infinity is a disjoint union of smooth compact submanifolds with no boundary in $\s^n$. Then there is a number $t_0 > 0$ such that the 
hypersurfaces $\phi_t$ are embedded for all $t \geq t_0$. 
\end{theorem}

\proof First, in the light of Lemma \ref{cpt-part}, one only needs to focus on each end. This is because, for the conformal metric $e^{2\rho}g_{\ts^n}$ corresponding 
to the given admissible hypersurface $\phi$, we know $\rho\to\infty$ when approaching $\partial\Omega$. Hence \eqref{continuity} holds, which implies the ends of hypersurface are well separated near the infinity.\\

Consider one of the connected components ${\rm E}^k\subset \partial\Omega$, which is a smooth compact submanifold with no boundary in $\s^n$. Let  $U$ be an open neighborhood of ${\rm E}^k$ in 
$\s^n$ whose closure $\bar U$ is a compact subset in $\s^n$ intersecting no other component of $\partial\Omega$. 
We denote the tubular neighborhood of $E^k$ inside $U$ with size $\lambda_0$ in $\s^n$ as 
$$
{\rm N}_{\lambda_0}({\rm E}^k) = {\rm E}^k\times B_{\lambda_0}^{n-k} \subset\s^n,
$$
where $B_\lambda^{n-k}$ is geodesic ball in $\s^{n-k}$ for $n-k\geq 1$ and $\lambda_0$ is some small positive number. Let $\ts^{n-k-1}_\lambda\subset B_{\lambda_0}^{n-k}$
denote the family of round spheres with $\lambda< \lambda_0$ and centered at the center of $B_{\lambda_0}^{n-k}$. Finally, for a point $p\in{\rm E}^k$, let $D^{n-k}_\lambda(p)\subset\h^{n+1}$ 
denote the totally geodesic hyperbolic subspace that has the boundary $\{p\}\times\ts^{n-k-1}_\lambda\subset{\rm N}_{\lambda_0}({\rm E}^k)$ at infinity. \\

From Lemma \ref{cpt-part} we know that there exists $t_0$ large enough so that the hypersurface 
$$
\phi_t: \Omega \bigcap (\bar U\setminus {\rm N}_{\frac 12\lambda_0}({\rm E}^k))\To\h^{n+1}
$$
is embedded and the hypersurface 
$$
\phi_t: \Omega\bigcap ({\rm N}_{\frac 12\lambda_0}({\rm E}^k)\setminus{\rm E}^k)\To\h^{n+1}
$$ 
lies inside $\bigcup_{p\in {\rm E}^k}\bigcup_{\lambda < \frac 23\lambda_0} D_{\lambda}^{n-k}(p)$ for each $t\geq t_0$.
Now let us consider the intersection $I_{p,\lambda}^t = D^{n-k}_\lambda(p)\cap\phi_t(\Omega)$ for each $p\in{\rm E}^k$ and $\lambda\leq \lambda_0$. First of all, one sees that each $I_{p, \lambda}^t$ is non-empty for $\lambda < \lambda_0$. This is a consequence of the fact that ${\rm E}^k$ is linked with each $\ts^{n-k-1}_\lambda$ in $\s^n$ when $\lambda$ is appropriately small, so ${\rm E}^k$ is still linked with $D^{n-k}_\lambda$ in the ball $\b^{n+1}$. It is then clear that 
$I_{p, \lambda}^t$ is a connected, embedded convex ovaloid ($n-k\geq 3$),  or a simple closed convex curve ($n-k = 2$), or a single point ($n-k = 1$), in a 
totally geodesic hyperbolic subspace $D^{n-k}_\lambda (p)$ when $\lambda \in (\frac 23\lambda_0, \lambda_0)$ and $t\geq t_0$. 
Another simple observation is the fact that each intersection $I_{p, \lambda}^t$ is compact, since the boundary at infinity $\ts^{n-k-1}$ of $D^{n-k}_\lambda(p)$ does not intersect with the 
boundary at infinity  $\partial\Omega$ of the hypersurface $\phi_t$. Our theorem is true if, for any given $t\geq t_0$, we are able to show that each $I_{p, \lambda}^t$ is a connected, 
embedded, convex ovaloid ($n-k\geq 3$),  or a simple closed convex curve ($n-k =2$), or a single point ($n-k =1$)  for all $p\in {\rm E}^k$, $\lambda < \lambda_0$. \\

Let us first establish the cases $k=0$, i.e. ${\rm E}^k$ is a point $p\in \s^n$. We observe that the intersection of a complete, strictly convex, immersed, hypersurface and a totally geodesic hyperbolic hyperplane can only be a union of connected, convex, immersed hypersurfaces and possibly finitely many other points in the hyperbolic hyperplane. This is because a strictly convex hypersurface 
is either transversal to a totally geodesic hyperbolic hyperplane or locally stays strictly on one side of the totally geodesic hyperbolic hyperplane at the intersection point. \\

 We claim that for any $\lambda<\lambda_0$ and $t\geq t_0$ each intersection $I_{p, \lambda}^t$ is a connected, immersed, compact, convex hypersurface ($n \geq 3$) or closed convex curve ($n=2$) in the totally geodesic hyperbolic hyperplane $D_\lambda^{n}(p)$. Hence our theorem, when $k=0$, follows from this claim and Theorem \ref{do-warner}.  Note that, in case $n = 2$, we instead use the fact that a connected, immersed, convex, closed curve is embedded if it is a limit of connected, embedded, convex, closed curves.\\

It is clear that $I_{p, \lambda}^t$ is a connected, compact, embedded, convex ovaloid when $\lambda$ is close to $\lambda_0$. This convex ovaloid stays as an embedded convex
ovaloid before some points emerge in $I_{p, \lambda}^t$ as $\lambda$ decreases from $\lambda_0$ in the light of Theorem \ref{do-warner}. 
To show that no point ever emerges in $I_{p, \lambda}^t$ we may assume 
otherwise $I_{p, \lambda_1}^t$ contains a point $q$ for the first time as $\lambda$ decreases from $\lambda_0$. One sees that the hyperbolic hyperplane $D^{n}_{\lambda_1}(p)$ 
is a support hyperplane at $q$ for the hypersurface $\phi_t$. It is clear that near $q$, the hypersurface $\phi_t$ lies locally on the side of the hyperplane $D_{\lambda_1}^n(p)$ that contains $p$ and the normal to the hypersurface $\phi_t$ at $q$ points to the same side due to the regularity of the Gauss map at infinity. But that would contradict \eqref{eq:strong-convex}. Therefore our theorem is proven when $k=0$. \\

The other extremal case is $k = n-1$. In this case $D^1_\lambda(p)$ is a geodesic with ends $(p, -\lambda)$ and $(p, \lambda)$ in ${\rm N}_{\lambda_0}({\rm E})\subset\s^n$. Instead of using hyperbolic hyperplanes we consider the ruled hypersurface  $\Sigma_\lambda = \bigcup_{p\in {\rm E}} D^1_\lambda (p)$ and the intersection 
$I_\lambda^t = \bigcup_{p\in {\rm E}}I_{p, \lambda}^t$. Assume otherwise, that for the first time, for some $\lambda_1$, among all $p\in {\rm E}$ and $\lambda$ decreasing from $\lambda_0$, the
intersection $I_{p, \lambda_1}^t$ contains more than just a single point. Then the hypersurface $\phi_t$ at the touch point, which has just emerged in $I^t_{p, \lambda_1}$, is tangent to the ruled hypersurface $\Sigma_{\lambda_1}$ and would possess some principal curvature nonnegative, which contradicts \eqref{eq:strong-convex}, similar to the situation dealt in the case $k=0$. \\

For general $k$ between $0$ and $k-1$, we are going to use the combination of the above two special cases. We claim that each $I_{p, \lambda}^t$ is an embedded, 
convex ovaloid or a simple closed convex curve in the totally geodesic hyperbolic subspace $D^{n-k}_\lambda (p)$. Again we will use  
the ruled hypersurface 
$$
\Sigma_\lambda^k = \bigcup_{p\in {\rm E}^k} D^{n-k}_\lambda (p),
$$ 
instead of hyperbolic hyperplanes. Assume otherwise, that for the first time, for some $\lambda_1$, among all $p\in {\rm E}^k$ and $\lambda$ decreasing from $\lambda_0$, some points emerge in the
intersection $I_{p, \lambda_1}^t$ other than the connected immersed convex surface ($n-k\geq 3$) or the connected convex closed curve ($n-k=2$). Then,  similar to the case $k = n-1$ above, the hypersurface $\phi_t$ at the touch point, which has just emerged in $I^t_{p, \lambda_1}$,  is tangent to the ruled hypersurface $\Sigma^k_{\lambda_1}$ so would have some principal curvature nonnegative, which contradicts \eqref{eq:strong-convex}. This completes the proof.

\endproof

\begin{remark}\label{end-structure} It is worth mentioning that the argument above is local in the sense that each component ${\rm E}^k$ is investigated independently. In other words, one may conclude that for $t$ large enough the hypersurface $\phi_t$ is embedded near those ends which are of manifold structure. 
It is also worth mentioning that, in fact, with the above argument we have shown that each end has the structure 
$$
{\rm E}^k\times \s^{n-k-1}\times (0, \infty),
$$
where $\s^{n-k-1}$ stands for a single point for $k = n-1$.
\end{remark}
 

\section{Elliptic Problems}\label{Sect:Ellipticity}

In this section we compare the elliptic problems associated with Weingarten hypersurfaces in hyperbolic space $\h^{n+1}$ to those of conformal metrics on domains of the conformal infinity $\s^n$. Both subjects have a long history and have been extensively studied. Although they are mostly treated separately, there is a clear indication that these two subjects should be intimately related in terms of the types of problems and the tools that have been used to study them. Our work here is an attempt to give a unified framework for the two subjects with a hope that in doing so, it will shed light on further investigation and research. For instance, comparing Obata type theorems and Alexandrov type theorems, we derive a new Alexandrov type theorem, which does not assume the hypersurface to be embedded. Similarly, comparing Liouville type theorems and Bernstein type theorems, we also obtain some new results.  \\

Unfortunately, the choices of convenient orientation between the discussions of admissible hypersurfaces and the elliptic problems of Weingarten hypersurfaces are opposite to each other. From here on we will take the orientation opposite to the canonical orientation for admissible hypersurfaces and we will refer to such orientation simply as the opposite orientation. 


\subsection{Corresponding elliptic problems }\label{elliptic correspondence}

For a comprehensive introduction of conformally invariant elliptic PDE we refer readers to the papers  \cite{Li, GS, S, LiLi1,LiLi2} and references therein. We will briefly introduce the conformally invariant elliptic PDE in the context of our discussions.  Since we focus on realizable conformal metrics, we denote
$$
{\cal C} := \{(x_1, \cdots , x_n) \in \r ^n : x_i < 1/2, i = 1, \cdots, n\}
$$ 
and
$$
\Gamma_n := \{(x_1,\cdots, x_n ): x_i >0, i=1, 2, \cdots, n\}.
$$ 
Consider a symmetric function $f(x_1, \cdots, x_n)$ of $n$-variables with $f(\lambda_0, \lambda_0, \cdots, \lambda_0) = 0$ for some number $\lambda_0 < \frac 12$ and
$$
\Gamma = \ \text{an open connected component of }\{(x_1, \cdots, x_n): f(x_1, \cdots, x_n) > 0\}
$$
satisfying

\begin{equation}\label{w1}
(\lambda, \lambda, \cdots, \lambda) \in \Gamma \bigcap {\cal C}, \forall \ \lambda \in (\lambda_0, \frac 12),
\end{equation}

\begin{equation}\label{w2}
\aligned
& \text{ $\forall \ (x_1, \cdots, x_n)\in \Gamma\cap{\cal C}$, $\forall \ (y_1, \cdots, y_n)\in\Gamma\cap{\cal C}\cap ((x_1, \cdots, x_n) + \Gamma_n)$, $\exists$ a curve $\gamma$} \\ & \text{ connecting $(x_1, \cdots, x_n)$ to $(y_1, \cdots, y_n)$ inside $\Gamma\cap{\cal C}$ such that $\gamma' \in \Gamma_n$ along $\gamma$,}
\endaligned
\end{equation}
and
 \begin{equation}\label{w3}
 f \in C^1(\Gamma) \ \text{and} \ \frac{\partial f }{ \partial x_i} >0 \text{ in } \Gamma.
 \end{equation}
Suppose $g = e^{2\rho}g_{\ts^n}$ is a conformal metric on a domain $\Omega$ of $\s^n$ satisfying

\begin{equation}\label{equ:elliptic}
f(\lsch) = C  \ \text{and} \ \lsch \in \Gamma \ \text{in} \ \Omega,
\end{equation}
for some positive constant $C$, where $\lsch$ is the set of eigenvalues of the Schouten curvature tensor of the metric $g$. In \eqref{equ:elliptic}, a positive constant $C$ is admissible for a given curvature function $f$ if $f(\bar\lambda_0, \bar\lambda_0, \cdots, \bar\lambda_0) = C$, $\frac{\partial f}{\partial x_i} (\bar\lambda_0, \bar\lambda_0, \cdots, \bar\lambda_0) > 0$, and $\bar\lambda_0 > \lambda_0$. We refer to equation \eqref{equ:elliptic} as the conformally invariant elliptic problem of the conformal metrics on the domain $\Omega$. \\

On the other hand, we have the following elliptic problems of Weingarten hypersurfaces. For a comprehensive introduction of Weingarten hypersurfaces we refer to the papers  \cite{EH,Ge,Ge2, Ko} and references therein. We will briefly introduce the elliptic problems of Weingarten hypersurfaces in our context. Again, our focus is on admissible hypersurfaces with the opposite orientation.  Let 
$$
{\cal K} := \{(x_1,\cdots, x_n)\in \r^{n} : x_i > -1, i = 1, \cdots, n\}.
$$
Consider a symmetric function ${\cal W}(x_1, \cdots, x_n)$ of $n$-variables with ${\cal W}(\kappa_0, \kappa_0, \cdots, \kappa_0) = 0$ for some number $\kappa_0 > -1$ and
$$
\Gamma^* =  \text{ an open connected component of } \{(x_1, \cdots, x_n): {\cal W}(x_1, \cdots, x_n) > 0\}
$$
satisfying

\begin{equation}\label{w4}
(\kappa, \kappa, \cdots, \kappa) \in  \Gamma^*\bigcap {\cal K}, \forall \ \kappa \in (\kappa_0, \infty),
\end{equation}

\begin{equation}\label{w5}
\aligned
& \text{ $\forall \ (x_1, \cdots, x_n)\in \Gamma^*\cap{\cal K}$, $\forall \ (y_1, \cdots, y_n)\in\Gamma^*\cap{\cal K}\cap ((x_1, \cdots, x_n) + \Gamma_n)$, $\exists$ a curve } \\ & \text{$\gamma$ connecting  $(x_1, \cdots, x_n)$ to $(y_1, \cdots, y_n)$ inside $\Gamma^*\cap{\cal K}$ such that $\gamma' \in \Gamma_n$ along $\gamma$,}
\endaligned
 \end{equation}
and
 \begin{equation}\label{w6}
 {\cal W} \in C^1 (\Gamma^*) \text{ and }\frac{\partial {\cal W} }{ \partial x_i} >0 \text{ in } \Gamma^*.
 \end{equation}
Suppose $\phi: \rm{M}\to \h^{n+1}$ is a hypersurface satisfying

\begin{equation}\label{equ:weingarten}
{\cal W}(\kappa_1, \cdots, \kappa_n) = K \ \text{and} \ (\kappa_1, \cdots, \kappa_n)\in \Gamma^* \ \text{on} \ \phi,
\end{equation}
for some positive constant $K$, where $(\kappa_1, \cdots, \kappa_n)$ is the set of principal curvatures of the hypersurface $\phi$. In \eqref{equ:weingarten}, a positive number $K$ is admissible for a given curvature function $\cal W$ if ${\cal W }(\bar\kappa_0, \bar\kappa_0, \cdots, \bar\kappa_0) = K, \frac {\partial {\cal W}}{\partial x_i} (\bar\kappa_0, \bar
\kappa_0, \cdots, \bar\kappa_0) > 0$, and $\bar\kappa_0 > \kappa_0$. We refer to equation \eqref{equ:weingarten} as the elliptic problem of Weingarten hypersurfaces.

\begin{remark} For the motivation of \eqref{w2} and \eqref{w5}, please see the proof of Theorem \ref{Th:GeneralizedBernstein}, where \eqref{w5} is shown to be sufficient to 
apply the Alexandrov reflection method. On the other hand, it is more appropriate to use curves instead of rays in \eqref{w2} and \eqref{w5}, since the curvature relation is non-linear.
\end{remark}

To relate these two elliptic problems, in the light of Theorem \ref{Th:representacion}, we consider 

\begin{equation}\label{map-t}
{\cal T} (x_1,\cdots, x_n)= \left(\frac 12 -  \frac 1{1 +x_1}, \cdots, \frac 12 - \frac 1{1+ x_n}\right): \cal K \to \cal C.
\end{equation}
Let us discuss the correspondence between conformally invariant elliptic problems of realizable metrics and elliptic problems of admissible Weingarten hypersurfaces. By our definitions, only $\Gamma\cap{\cal C}$ is relevant for a realizable metric and only $\Gamma^*\cap{\cal K}$ is relevant for an admissible hypersurface with the opposite orientation. Below we list some fundamental relations 
and facts for the correspondence between the elliptic problems of conformal metrics and Weingarten hypersurfaces. 

\vskip 0.1in\noindent{\bf Symmetric Functions}:

\begin{quote} 
\begin{equation}\label{function}
{\cal W}_f = f\circ {\cal T} \ \text{and} \  {f}_{\cal W} = {\cal W}\circ {\cal T}^{-1}.
\end{equation}
\end{quote}

\vskip 0.1in\noindent{\bf Domains}:

\begin{quote} 
\begin{equation}\label{cone}
{\cal T}(\Gamma^*\cap{\cal K}) = \Gamma\cap\cal C.
\end{equation}
\end{quote}
It is clear that
\begin{equation}\label{gamma-n}
{\cal T} ((\kappa_0, \kappa_0, \cdots, \kappa_0) + \Gamma_n) = ((\lambda_0, \lambda_0, \cdots, \lambda_0) + \Gamma_n)\cap \cal C,
\end{equation}
where $\lambda_0 = \frac 12 - \frac 1{ 1 + \kappa_0}$. In fact,
\begin{equation}\label{gamma-n+}
{\cal T}((x_1, \cdots, x_n) + \Gamma_n) = ({\cal T}(x_1, \cdots, x_n) + \Gamma_n)\cap{\cal C}
\end{equation}
for all $(x_1, \cdots, x_n)\in {\cal K}$. Therefore \eqref{w1} holds for $\Gamma\cap{\cal C}$ if and only if \eqref{w4} holds for $\Gamma^*\cap{\cal K}$. Moreover,  
we also see \eqref{w2} holds for $\Gamma\cap{\cal C}$ if and only if \eqref{w5} holds for $\Gamma^*\cap{\cal K}$.

\vskip 0.1in\noindent{\bf Ellipticity}:
\begin{quote}
\begin{equation}\label{ellipticity}
\aligned
& \text{$\frac{\partial {\cal W}_f}{\partial \kappa_i} > 0$ in $\Gamma^*\cap{\cal K}$ if and only if $\frac{\partial f}{\partial \lambda_i} > 0$ in $\Gamma\cap{\cal C}$.}\\
&\text{$\frac{\partial f_{\cal W}}{\partial \lambda_i} > 0$ in $\Gamma\cap{\cal C}$ if and only if $\frac{\partial {\cal W}}{\partial \kappa_i} > 0$ in $\Gamma^*\cap{\cal K}$.}
\endaligned
\end{equation}
\end{quote}

\noindent{\bf Homogeneity}: \vskip 0.1in

\begin{quote}
Homogeneity of symmetric functions is not preserved under this correspondence. In fact, scaling on the metric side corresponds to deforming along the normal flow in hypersurface side.
 \end{quote}
 
\vskip 0.1in\noindent{\bf Concavity}:\vskip 0.1in

\begin{quote}
The concavity on the other hand is preserved under this correspondence from $f$ to $\cal W$, but not necessarily from $\cal W$ to $f$. The concavity of a function is understood to be the 
nonpositivity of the Hessian matrix. One may simply calculate that

$$ \dfrac{\partial ^2 \mathcal W_f}{\partial \kappa_i \partial \kappa_j} = \frac{1}{(1+\kappa_i)^2 (1+\kappa_j)^2}  \dfrac{\partial ^2 f}{\partial \lambda _i \partial \lambda _j}  -\frac{2 \delta _{ij}}{(1+\kappa_i)^3} \dfrac{\partial f}{\partial \lambda _i} $$
and 
$$ \dfrac{\partial ^2 f_{\cal W}}{\partial \lambda_i \partial \lambda_j}  = \frac{1}{(\frac 12 - \lambda_i)^2 (\frac 12 - \lambda_j)^2}  \dfrac{\partial ^2 {\cal W}}{\partial \kappa _i \partial \kappa _j}+  \frac{2 \delta _{ij}}{(\frac 12 - \lambda_i)^3} \dfrac{\partial {\cal W}}{\partial \kappa _i} .
$$
Hence, instead, the convexity is preserved under this correspondence from ${\cal W}$ to $f$.
\end{quote}

\newpage\noindent{\bf Admissible Constants}:\vskip 0.1in

\begin{quote}
In \eqref{equ:elliptic}, a positive constant $C$ is admissible for a given curvature function $f$ if $f(\bar\lambda_0, \bar\lambda_0, \cdots, \bar\lambda_0) = C$, 
$\frac{\partial f}{\partial x_i} (\bar\lambda_0, \cdots, \bar\lambda_0) > 0$, and $\bar\lambda_0 > \lambda_0$, while in \eqref{equ:weingarten}, a positive constant $K$ is admissible for a given curvature function $\cal W$ if ${\cal W }(\bar\kappa_0, \cdots, \bar\kappa_0) = K, \frac {\partial {\cal W}}{\partial x_i} (\bar
\kappa_0, \cdots, \bar\kappa_0) > 0$, and $\bar\kappa_0 > \kappa_0$, where $\bar\lambda_0 = \frac 12 - \frac 1{1 + \bar\kappa_0}$. Geometrically it means that  the horospherical 
metric of a geodesic sphere of principal curvature $\kappa > 1$ is of constant sectional curvature $<1$; the horospherical metric of a horosphere is of zero sectional curvature; and
the horospherical metric of a hypersphere of principal curvature $\kappa\in [0, 1)$ is of negative constant sectional curvature. 
\end{quote}

\vskip 0.1in\noindent{\bf Scalar Curvature vs Mean Curvature}:

\begin{quote} 
In the context of solving elliptic problems one typically assumes
$$
\Gamma\subset \Gamma_1= \{(x_1, \cdots, x_n): \sum_{i=1}^n x_i \geq 0\}
$$
and
$$
\Gamma^* \subset \Gamma^*_1=\{(x_1, \cdots, x_n): \sum_{i=1}^n x_i \geq n\}.
$$
In contrast to ${\cal T}(\Gamma^*_n) = \Gamma_n\cap{\cal C}$, we only have
\begin{equation}\label{scalar-mean}
{\cal T}^{-1}(\Gamma_1) \subset \Gamma^*_1.
\end{equation}
Therefore, we only have $\Gamma\cap{\cal C}\subset \Gamma_1$ implies $\Gamma^*\cap{\cal K}\subset \Gamma^*_1$, but not necessarily the converse.
\end{quote}

Before we end this subsection we would like to give a proof of \eqref{scalar-mean}. This turns out to be a consequence of the following simple algebraic fact.

\begin{lemma}\label{H-R}
Let $a_i $ be real numbers so that $a_i > -1$ for all $i =1, \cdots , n$. Set $b_i = \dfrac{a_i -1}{a_i +1}$. Then 
$$\sum _{i=1}^n b_i \leq 2 \sum _{i=1}^n a_i - n .$$
\end{lemma}
\begin{proof}
Set $Q(t)=  \sum _{i=1}^n \dfrac{a_i}{1 + a _i t}$ for $0 \leq t \leq 1$. Then, 
$$ 
Q' (t) = - \sum _{i=1}^n \frac{a_i ^2}{ (1+ a_i t)^2} \leq 0, 
 $$
which implies $  Q(1) \leq Q(0)=\sum _{i=1}^n a_i $. On the other hand,
\begin{equation*}
\begin{split}
\sum _{i=1}^n b_i &= \sum _{i=1}^n \frac{a_i -1}{a_i +1 }= 2\sum _{i=1}^n \frac{a_i}{a_i +1} - \sum _{i=1}^n \frac{a_i +1}{a_i +1}\\[3mm]
 &=  2 Q(1) - n.
\end{split}
\end{equation*}
Therefore, the lemma is easily seen.
\end{proof}

Let $\phi: {\rm M}^n\to\h^{n+1}$ be a horospherically convex hypersurface with the orientation opposite to the canonical one.  Then, $\lambda _i $ and $\kappa _i$ are related by 
\begin{equation}\label{1}
 2 \lambda _i = \frac{\kappa _i  -1}{\kappa _i +1} ,
\end{equation}or equivalently, 
\begin{equation}\label{2}  
\kappa _i = \frac{1 + 2 \lambda _i}{1-2 \lambda _i}.
\end{equation}

Set $a_i := -2 \lambda _i$ and $b_i := - \kappa _i$. Note that since $\lambda _i < 1/2$, then $1-2 \lambda _i >0$, that is, $a_i > -1$. Therefore, from \eqref{1} it follows
$$ -\sum _{i=1}^n \kappa _i \leq -4 \sum _{i=1}^n \lambda_i - n,$$
so that
\begin{equation}\label{eq:H-R} 
\sum _{i=1}^n \lambda _i \geq 0 \quad \text{ implies } \quad \sum _{i=1}^n \kappa _i \geq n, 
\end{equation}
which in turn implies \eqref{scalar-mean}.


\subsection{Obata Theorem  vs Alexandrov Theorem}\label{obata-alexandrov}

Here we establish an explicit relationship between a famous theorem in conformal geometry and a famous theorem in hypersurface theory. Namely, the Obata Theorem and the Alexandrov Theorem (cf. \cite{EGM}). First, let us state the aforementioned results. For conformal metrics, we have the following:

\begin{quote}
{\bf Obata Theorem \cite{Ob,GNN}} {\it Let $g$ be a metric conformal to the standard round metric $g_{\ts^n}$ on $\s ^n$ with constant positive scalar curvature. Then, there exists a conformal diffeomorphism $\Phi : \s ^n \to \s ^n $ and a postive constant $a>0$ such that $g = a \, \Phi ^* g_{\ts^n}$.} 
\end{quote}
Its generalization to fully nonlinear elliptic functions $(f, \Gamma)$ is as follows:

\begin{quote}
{\bf Generalized Obata Theorem \cite{LiLi1}} {\it Let $g$ be a metric conformal to the standard round metric $g_{\ts^n}$ on $\s ^n$. Suppose that $(f, \Gamma)$ is  elliptic in the sense that it satisfies \eqref{w1} \eqref{w2}  \eqref{w3} and  that 
$$ f (\lsch) = C , \, \lsch \in \Gamma $$
for an admissible constant $C$. Then, there exists a conformal diffeomorphism $\Phi : \s ^n \to \s ^n $ and a positive constant $a$ such that $g = a \, \Phi ^* g_{\ts^n}$.}
\end{quote}
For hypersurfaces in hyperbolic space we have the following:

\begin{quote}
{\bf Alexandrov Theorem \cite{A,Ho}} {\it Let $\Sigma \subset \h ^{n+1}$ be a compact (without boundary) embedded hypersurface with constant mean curvature. Then, $\Sigma$ is a totally umbilical round sphere.}
\end{quote}
Its generalization to elliptic Weingarten hypersurfaces $({\cal W}, \Gamma^*)$ is as follows:

\begin{quote}
{\bf Generalized Alexandrov Theorem \cite{Ko}} {\it Let $\Sigma\in \h^{n+1}$ be compact (without boundary) embedded hypersurface. Suppose that $({\cal W}, \Gamma^*)$ satisfies 
\eqref{w4} \eqref{w5} \eqref{w6} and that 
$$
{\cal W} (\kappa_1, \cdots, \kappa_n) = K \ \text{and} \ (\kappa_1, \cdots, \kappa_n)\in \Gamma^*
$$
on $\Sigma$, where $K$ is admissible. Then $\Sigma$ is totally umbilical round sphere.}
\end{quote}

In the light of the correspondence observed in Theorem \ref{Th:representacion} of  \cite{EGM} and the discussions in Section \ref{elliptic correspondence} we obtain a new Alexandrov type theorem for horospherical ovaloids as an equivalent statement of the generalized Obata Theorem of Li-Li \cite{LiLi1, LiLi2} above (notice that a conformal metric on $\s^n$ is always realizable).  But due to our 
Corollary \ref{cpt-embedding}, it can be seen as a consequence of the generalized Alexandrov Theorem of Korevaar \cite{Ko}. Thus, such new Alexandrov type theorem becomes the bridge connecting the two sides and it is interesting to see that the generalized Alexandrov Theorem of Korevaar implies the generalized Obata Theorem of Li-Li
\cite{LiLi1, LiLi2}, instead of the other way around as given in \cite{EGM}. 

\begin{theorem}\label{Th:GeneralizedAlexandrov}
Suppose that $({\cal W}, \Gamma^*)$ satisfies 
\eqref{w4} \eqref{w5} \eqref{w6}. Then a horospherical ovaloid in $\h^{n+1}$ with the opposite orientation satisfying \eqref{equ:weingarten} for an admissible constant is a geodesic 
sphere in $\h^{n+1}$. Equivalently, suppose that $(f, \Gamma)$ satisfies \eqref{w1} \eqref{w2} \eqref{w3}. Then any conformal metric on $\s^n$ satisfying \eqref{equ:elliptic} for an
admissible constant is isometric to a round metric on $\s^n$.
\end{theorem}

\begin{proof} 
According to Theorem \ref{Th:Correspondence}, the horospherical ovaloid $\Sigma_t$ along the normal flow of the given horospherical ovaloid $\Sigma$ becomes embedded when $t\geq t_0$ for some $t_0$.  Moreover,
$$ 
\kappa _i = \dfrac{\kappa_i ^t - \tanh (t)}{1- \kappa ^t_i \tanh (t)} \ \text{ and } \ \kappa^t_i = \dfrac{\tanh(t) + \kappa_i}{1 +\kappa_i\tanh(t)}.
$$
Hence, $\Sigma _t$ is still an elliptic Weingarten hypersurface for all $t > t_0$. To see this, we let 
$$
\mathcal{W}^t (x_1 , \cdots , x_n) := \mathcal W \left( \dfrac{x_1- \tanh (t)}{1- x_1 \tanh (t)} , \cdots , \dfrac{x_n - \tanh (t)}{1- x_n \tanh (t)}\right).
$$ 
Therefore, ${\cal W}^t$ is a symmetric function of $n$-variables with ${\cal W}^t(1, \cdots, 1)= 0$.  Let 
$$
T(x_1, \cdots, x_n) = \left(\dfrac{\tanh(t) +x_1}{1+x_1\tanh(t)}, \cdots, \dfrac{\tanh(t) + x_n}{1 +x_n\tanh(t)}\right).
$$
We then have
$$
\Gamma^*_t \cap{\cal K}= T(\Gamma^*\cap{\cal K}).
$$
Similar to the case of the map $\cal T$, we in fact have
$$
T((x_1, \cdots, x_n) + \Gamma_n) = T(x_1, \cdots, x_n) + \Gamma_n
$$
for all $(x_1, \cdots, x_n)\in \cal K$. For ellipticity we easily calculate
$$
\dfrac{\partial {\cal W}^t}{\partial x_i} = \dfrac{1 - \tanh^2(t)}{(1 + x_i\tanh(t))^2} \dfrac{\partial {\cal W}}{\partial y_i}.
$$
Therefore, $({\cal W}^t, \Gamma^*_t)$ satisfies \eqref{w4} \eqref{w5} \eqref{w6}.  Thus, from the above generalized Alexandrov Theorem of Korevaar \cite{Ko}, $\Sigma _t$ is a totally umbilical round sphere for $t\geq t_0$,  and therefore so is $\Sigma$.
\end{proof}


\subsection{Liouville Theorem vs Bernstein Theorem}

Next to compact hypersurfaces in $\h^{n+1}$, the simplest noncompact hypersurfaces in $\h^{n+1}$ are those that have a single point at the infinity $\s^n$. Their corresponding domains in $\s^n$ are punctured spheres $\s^n\setminus\set{\bf n}$. In this context, we establish an explicit relationship between another pair of celebrated theorems in conformal geometry and in hypersurface theory: Liouville type theorems and Bernstein type theorems. We focus on the cases where the positive constants in \eqref{equ:elliptic} and \eqref{equ:weingarten} are admissible and elliptic equations are non-degenerate. First, let us state the aforementioned results.

\begin{quote}
{\bf Liouville Theorem \cite{CGS}} {\it The only complete conformal metrics on $\s ^n \setminus \set{\bf n}$ with nonnegative constant scalar curvature is the Euclidean metric.}
\end{quote}
Its generalization  to fully nonlinear non-degenerate elliptic functions $(f, \Gamma)$ is as follows:

\begin{quote}
{\bf Generalized Liouville Theorem \cite{LiLi2}} {\it Suppose that $(f, \Gamma)$ is elliptic in the sense that it satisfies \eqref{w1} \eqref{w2}  \eqref{w3}. Then the only possible 
complete conformal metric of nonnegative scalar curvature on $\s^n\setminus\set{\bf N}$ satisfying \eqref{equ:elliptic} for an admissible constant is the Euclidean metric.}
\end{quote}
We remark that the above theorems are simplified versions of Theorem 1.4 in \cite{LiLi1} and Theorem 1.3 in \cite{LiLi2}. For hypersurfaces in hyperbolic space we have the following:

\begin{quote}
{\bf Bernstein Theorem \cite{B, dCL}:} {\it The only properly embedded, complete,  constant mean curvature $H\geq n$ hypersurfaces with one point at infinity in $\h ^{n+1}$ are horospheres.}
\end{quote}
As far as we know, the above result has been generalized for special Weingarten surfaces in $\h^{3}$ (see \cite{RSa,SaT} and \cite{AEG}), but not for higher dimensions. We follow the proof in \cite{dCL} and establish the following generalized Bernstein Theorem:

\begin{theorem}[Generalized Bernstein Theorem]\label{Th:GeneralizedBernstein}
Suppose that $({\cal W}, \Gamma^*)$ is an elliptic function satisfying \eqref{w4} \eqref{w5} \eqref{w6}. Then the only possible properly embedded, complete hypersurface in $\h^{n+1}$ 
satisfying \eqref{equ:weingarten} for an admissible constant with only one point at infinity is a horosphere. 
\end{theorem}

\begin{proof} The proof is more or less the same as the proof of Theorem A given in Section 2 of \cite{dCL}. The readers are referred to \cite{dCL} for more details. It is particularly helpful
to use the Figure 1 and 2 in Section 2 of \cite{dCL}. However, we would like to take this opportunity to clarify that our assumptions \eqref{w4} \eqref{w5} \eqref{w6} are sufficient for the argument. \\

Suppose that there is such surface $\Sigma$ in $\h^{n+1}$. We use the half space model for hyperbolic space $\h^{n+1}$ so that the single point infinity of the surface $\Sigma$ is at the infinity of $\r^{n+1}_+$. Let $\gamma$ be any vertical line in the half space $\r^{n+1}_+$.  Then $\gamma$ is a complete geodesic in $\h^{n+1}$. Let $P_t$ denote the foliation of totally geodesic hyperplanes orthogonal to $\gamma$ passing through $\gamma (t)$. Since $\Sigma $ is properly embedded, there exists $t_0$ so that $\Sigma \cap P_t = \emptyset$ for all $t < t_0$ and $t_0$ is the first time that $P_t$ touches $\Sigma$. At this first point of contact, $\Sigma$ is locally a graph over $P_{ t_0}$. We can raise $P_t$ further for $t> t_0$ to obtain $\Sigma _-(t):= \Sigma \cap \bigcup_{s< t} P_s$, which is a graph under $P_{t}$ at least for $t$ close to $t_0$. Now we reflect $\Sigma _-(t)$ with respect to $P_{t}$ and denote the reflection by $\tilde \Sigma _-(t)$. The proof of Theorem A in \cite{dCL} is based on the fact that the reflection $\tilde\Sigma_-(t)$ can never touch the rest of the surface $\Sigma_+(t) :=\Sigma\setminus\Sigma_-(t)$, not even at the boundary of  the reflection $\tilde\Sigma_-(t)$. Here, by touch, we mean their tangent hyperplanes are parallel at the point of contact. \\

Assumption \eqref{w4} allows us to define admissible constants $K$. Assume otherwise the reflection $\tilde\Sigma_-(t)$ does touch the rest $\Sigma_+(t)$. When $\tilde\Sigma_-(t)$ first touches $\Sigma_+(t)$, with respect to the inward normal to $\Sigma$, $\tilde\Sigma_-(t)$ is above $\Sigma_+(t)$.  Hence the principal curvatures $(\kappa^-_1, \cdots, \kappa^-_n)$ of $\tilde\Sigma_-(t)$ at the touch point  are in $\Gamma^*\cap((\kappa_1, \cdots, \kappa_n)+\Gamma_n)$, where $(\kappa_1, \cdots, \kappa_n)$ are the 
principal curvatures of $\Sigma_+(t)$ at the touch point. Then \eqref{w5} and \eqref{w6} imply $(\kappa^-_1, \cdots, \kappa^-_n) = (\kappa_1, \cdots, \kappa_n)$ due to that fact 
${\cal W}(\kappa^-_1, \cdots, \kappa^-_n) = {\cal W}(\kappa_1, \cdots, \kappa_n)$. After we know the principal curvatures are the same for the two surfaces at the touch point, we may consider a smooth family of surfaces deforming from $\tilde\Sigma_-(t)$ to the $\Sigma_+(t)$, where the function $\cal W$ is well defined, at least locally near the touch point. Then the Hopf maximum principle is applicable and gives the fundamental proposition of the Alexandrov reflection method due to the assumption \eqref{w6} (cf. \cite{Ko}). \\

Having explained the use of our assumptions \eqref{w4} \eqref{w5} \eqref{w6}, we now briefly recapture the idea in the proof of Theorem A in \cite{dCL}. First one proves that the reflection $\tilde\Sigma_-(t)$ can never touch the rest of the surface $\Sigma_+(t)$. Second one observes that if the surface is not a horosphere in $\h^{n+1}$ (i.e. a horizontal hyperplane in $\r^{n+1}_+$), then the incident that the reflection touches the rest of the surface at the boundary for the first time always happens for some vertical line (cf. Figure 2 in Section 2 of \cite{dCL}). This completes the proof.  
\end{proof}

As a consequence of Theorem \ref{uniform-bounded} and Theorem \ref{Th:Correspondence}, we have the following:

\begin{corollary}\label{Cor:Sn-p}
Suppose that $({\cal W}, \Gamma^*)$ is an elliptic function satisfying \eqref{w4} \eqref{w5} \eqref{w6}. Then the only possible admissible hypersurface in $\h^{n+1}$ with the opposite orientation and a single point at infinity satisfying \eqref{equ:weingarten} for an admissible constant is a horosphere. \\

Equivalently, suppose that $(f, \Gamma)$ satisfies \eqref{w1} \eqref{w2} \eqref{w3}. Then the only possible realizable metric on $\s^n\setminus\set{p}$ satisfying \eqref{equ:elliptic} 
for an admissible constant is the Euclidean metric.
\end{corollary}


\subsection{General Symmetry}\label{Sect:Symmetry}

In this subsection we derive a slight extension of \cite[Theorem 2.1]{LR}. As a consequence, we will derive Delaunay type theorems for admissible hypersurafces as well as 
realizable metrics. \\

First we introduce some notation. Let us denote the group of conformal transformations on $\s^n$ by $\Diff$ and the group of isometries of  $\h^{n+1}$ by $\IsoH$. 
Let $g$ be a conformal metric on a domain $\Omega \subset \s ^n$. We say that $g$ is $\Phi -$invariant if
$$
\Phi:\Omega\to\Omega \text{ and }  g= \Phi ^* g 
$$
for $\Phi\in \Diff$.  And we say that a hypersurface $\Sigma$ in $\h^{n+1}$ is $I-$invariant if 
$$
I: \Sigma\to\Sigma
$$
for $I\in\IsoH$. Remember that a conformal transformation $\Phi \in \Diff$ induces an unique isometry $I\in \IsoH$ and vice versa (cf.  \cite{D}, for instance).  
The following fact can be verified and appeared in \cite{E}: 

\begin{lemma}\label{Lem:invariant}  Let $\phi: {\rm M}^n\to\h^{n+1}$ be an admissible hypersurface in
$\h^{n+1}$ and $g$ be corresponding realizable metric on $G(\rm{M})$. Let $I \in \IsoH $ be an isometry and $\Phi \in \Diff$ be the associated conformal
transformation. Then $\phi$ is $I-$invariant if and only if $g$ is $\Phi -$invariant.
\end{lemma}

From this fact we know that symmetries are preserved under the correspondence between admissible hypersurfaces and realizable metrics. Our issue here is to retain the symmetry for a complete hypersurface in $\h^{n+1}$ from that of its boundary at infinity or equivalently to retain the symmetry for a complete conformal metric on a domain in $\s^n$ from that of the domain. The proof of the following slight extension of \cite[Theorem 2.1]{LR} is readily seen from the original proof in \cite{LR} and from the proof of Theorem \ref{Th:GeneralizedBernstein}. To state our theorem we introduce some more notation. Let $E$ be the equator in $\s^n$ and $P$ be the totally geodesic hyperplane whose boundary is $E$. Let $R$ stand for the reflection in $\h^{n+1}$ with respect to the hyperplane $P$.

\begin{theorem}\label{Th:symmetry} Suppose that $({\cal W}, \Gamma^*)$ satisfies \eqref{w4} \eqref{w5} \eqref{w6}. Let $\Sigma\subset\h^{n+1}$ be a properly embedded hypersurface whose boundary  $\partial_\infty\Sigma$ at the infinity is in $E$.  Assume that $\Sigma$ is an elliptic Weingarten hypersurface in the sense that the equation \eqref{equ:weingarten} holds on $\Sigma$ for an admissible constant $K$. Then $\partial_\infty\Sigma$ can not be all of $E$ and the surface $\Sigma$ is $R-$invariant.
\end{theorem}

Again,  as a consequence, we conclude the following general Alexandrov reflection principle for both admissible hypersurfaces and realizable metrics:

\begin{corollary} \label{Cor:symmetry} Suppose that $({\cal W}, \Gamma^*)$ satisfies \eqref{w4} \eqref{w5} \eqref{w6}. Let $\phi: {\rm M}^n\to\h^{n+1}$ be an admissible hypersurface 
with the opposite orientation satisfying \eqref{equ:weingarten},  whose boundary $\partial_\infty\phi({\rm M}^n)$ at the infinity is a disjoint union of smooth compact submanifolds with no boundary in $E$.  Then $\partial_\infty\phi({\rm M}^n)$ can not be $E$ and the surface $\phi$ is $R-$invariant.  \\

Equivalently, suppose that $(f, \Gamma)$ satisfies \eqref{w1} \eqref{w2} \eqref{w3}. Let $g$ be a realizable metric satisfying \eqref{equ:elliptic} on $\Omega$ such that $\partial\Omega \subset E$ is a disjoint union of smooth compact submanifolds with no boundary. Then $\partial\Omega$ can not be $E$ and $g$ is $R-$invariant.
\end{corollary}

There are many consequences of Corollary \ref{Cor:symmetry}. In particular, when the boundary at infinity consists of exactly two points, we obtain the following Delaunay 
type theorem.  

\begin{corollary}\label{Cor:Delaunay}  Suppose that $({\cal W}, \Gamma^*)$ satisfies \eqref{w4} \eqref{w5} \eqref{w6}. Let $\phi: {\rm M}^n\to\h^{n+1}$ be an admissible hypersurface 
with the opposite orientation satisfying \eqref{equ:weingarten}, whose boundary  $\partial_\infty\phi({\rm M}^n)$ at the infinity consists of exactly two points.  Then the surface $\phi$ is rotationally symmetric with respect to the geodesic joining the two points at the infinity of $\phi$.  \\

Equivalently, suppose that $(f, \Gamma)$ satisfies \eqref{w1} \eqref{w2} \eqref{w3}. Let $g$ be a realizable metric satisfying \eqref{equ:elliptic} on $\Omega = \s^n\setminus\{p, q\}$.
Then $g$ is cylindric with respect to the geodesic joining the two points in $\partial\Omega$. 
\end{corollary}

In the theory of hypersurfaces in hyperbolic space, Delaunay type theorems were established in \cite{Hs,LR} for constant mean curvature surfaces, and in \cite{AEG,RSa,SaT} for special
Weingarten surfaces in $\h^3$. Also Corollary \ref{Cor:Delaunay} should be compared with Theorem 1.2 in \cite{Li}, where the scalar curvature is assumed to be nonnegative.


\end{document}